\documentclass[10pt,twoside,english,reqno,a4paper]{amsproc}

\newlength{\defbaselineskip}
\usepackage[titletoc,title]{appendix}
\usepackage{listings,graphicx,amsmath,varioref,amscd,amssymb,color,bm,stmaryrd}
\usepackage{mathpazo} 
\usepackage[scaled=.95]{helvet} 
\usepackage{courier}
\usepackage{amssymb}
\usepackage{amsmath}
\usepackage{amsthm}
\usepackage{amscd}
\usepackage{amssymb}
\usepackage{enumerate}
\usepackage{amsfonts}
\usepackage{graphics}
\usepackage{pgfplots}
\pgfplotsset{width=7cm,compat=newest} 
\usepackage{graphics}
\usepackage{tikz}
\usetikzlibrary{shapes.geometric}
\usepackage{tikz-cd}
\usepackage[english]{babel}

\usepackage{hyperref}
\hypersetup{
colorlinks=true,
linkcolor=red!70!black, 
citecolor= orange, 
urlcolor=blue, 
}

\usepackage[capitalize,nameinlink,noabbrev,nosort]{cleveref} 

\crefname{section}{Section}{Sections}
\crefname{subsection}{Subsection}{Subsections}
\crefname{appsec}{Appendix}{Appendices}
\crefname{appendix}{Appendix}{}
\crefname{equation}{}{}
\crefname{figure}{Figure}{Figures}

\crefname{definition}{Definition}{Definitions}
\crefname{theorem}{Theorem}{Theorems}
\crefname{proposition}{Proposition}{Propositions}
\crefname{corollary}{Corollary}{Corollaries}
\crefname{remark}{Remark}{Remarks}
\crefname{lemma}{Lemma}{Lemmas}

\crefname{theoa}{Theorem}{Theorems}
\crefname{propa}{Proposition}{Propositions}
\crefname{rmkappa}{Remark}{Remarks}
\crefname{lema}{Lemma}{Lemmas}

\usepackage{graphics}
\usepackage{pgfplots}
\pgfplotsset{width=7cm,compat=newest} 
\usepackage{graphics}
\usepackage{tikz}
\usetikzlibrary{shapes.geometric}
\usepackage{tikz-cd}
\usepackage{caption}
\usepackage{subcaption}

\usepackage{float}

\theoremstyle{plain}
\newtheorem{theorem}{\textbf{{Theorem}}}[section]

\newtheorem{proposition}[theorem]{\textbf{{Proposition}}}

\newtheorem{lemma}[theorem]{\textbf{{Lemma}}}

\newtheorem{corollary}[theorem]{\textbf{{Corollary}}}

\newtheorem{definition}{\textbf{{Definition}}}[section]
\newtheorem{remark}[theorem]{\textbf{{Remark}}}

\def\N{\nabla}
\def\vp{\varphi}
\newcommand{\integrale}{\int_\Omega}
\def\al{\alpha}
\def\ro{\rho}
\def\ga{\gamma}
\def\eps{\varepsilon}
\def\te{\theta}

\def\vep{\varepsilon}
\def\R{\mathbb{R}}
\def\tS{\tilde{S}}
\def\into{\int_\Omega}
\def\continue#1{C^0([0,T]; \elle{#1})} 
\def\limitate#1{L^\infty((0,T); \elle{#1})} 
\def\paraccaloc{L^2_{loc}((0,T); H^1_0(\Omega))}
\def\be{\begin{equation}}
\def\ee{\end{equation}}
\def\bes{\begin{equation*}}
\def\ees{\end{equation*}}
\def\bac{\begin{array}{c}}
\def\eac{\end{array}}
\def\beac{\be\begin{array}{c}}
\def\eeac{\end{array}\ee}
\def\besac{\bes\begin{array}{c}}
\def\eesac{\end{array}\ees}
\def\rife#1{(\ref{#1})}
\def\elle#1{L^{#1}(\Omega)}
\def\parelle#1{L^{#1}(Q_T)}
\def\dive{{\rm div}}
\def\vfi{\varphi}
\newcommand{\meas}{\text{meas}}

\newcommand{\om}{\omega}
\newcommand{\si}{\sigma}
\newcommand{\de}{\delta}
\newcommand{\ds}{\displaystyle}
\newcommand{\D}{\Delta}

\definecolor{seagreen}{rgb}{0.18, 0.55, 0.34}

\renewcommand{\theequation}{\thesection.\arabic{equation}}
\numberwithin{equation}{section}
\long\def\salta#1{\relax}

\makeatletter

\def\og{\leavevmode\raise.3ex\hbox{$\scriptscriptstyle\langle\!\langle$~}}
\def\fg{\leavevmode\raise.3ex\hbox{~$\!\scriptscriptstyle\,\rangle\!\rangle$}}

\author[M. Magliocca]{Martina Magliocca}

\address[M. Magliocca]{Dipartimento di Matematica, Universit\`a degli Studi Tor Vergata, Via della Ricerca Scientifica 1, 00133 Rome, Italy. 
	\\ \url{magliocc@mat.uniroma2.it}}
	
\author[A. Porretta]{Alessio Porretta}
\address[A. Porretta]{Dipartimento di Matematica, Universit\`a degli Studi Tor Vergata, Via della Ricerca Scientifica 1, 00133 Rome, Italy. 
	\\ \url{porretta@mat.uniroma2.it}}


\begin{document}

\title[Local and global time decay]{Local and global time decay for parabolic equations with\\ super linear
first order terms}

\begin{abstract}
We study a class of parabolic equations having first order terms with superlinear (and subquadratic) growth. The model problem is  the so-called viscous Hamilton-Jacobi equation  with superlinear Hamiltonian. We address the problem of having unbounded initial data and we develop a   local theory yielding well-posedness for initial data in the optimal Lebesgue space, depending on the superlinear growth. Then we prove regularizing effects, short and long time decay estimates of the solutions. 

Compared to previous works, the main novelty is that our results  apply to nonlinear operators with just  measurable and bounded coefficients, since we totally avoid the use of gradient estimates of higher order. By contrast we only rely on elementary arguments using equi-integrability, contraction principles  and truncation methods for weak solutions.     
\end{abstract}

\maketitle

\section{Introduction}

In this paper we wish to study the behavior of  parabolic equations perturbed by first order terms having superlinear growth. In what follows,   $\Omega$ is an open bounded subset of $\mathbb{R}^N$ and, for  $T>0$, we set $Q_T:= (0,T)\times \Omega$. The reference case to keep in mind is the perturbed heat equation
\be\label{vhj}
\begin{cases}
\begin{array}{ll}
\ds u_t -\Delta u = |Du|^q & \ds\text{in }Q_T,\\
\ds u=0  &\ds \text{on }(0,T)\times \partial \Omega,\\
\ds u(0,x)=u_0(x) &\ds \text{in } \Omega,
\end{array}
\end{cases}
\ee
with Dirichlet boundary conditions.
This model example, which is often named as {\it viscous Hamilton-Jacobi equation}, has been extensively studied by many authors who pointed out different features depending on the possibly superlinear growth $q>1$ of the right-hand side. If $u_0$ is $C^1$, the classical parabolic theory applies  and provides with  a unique classical solution, which  is global in time if $q\leq 2$. On the other hand,  a  gradient blow-up may happen in finite time if $q>2$ and, in this latter case, a global in time solution exists in the weaker formulation of viscosity solutions. We refer to \cite{BdL,Sou,Sou-Z},  for this kind of analysis. For all $q>1$, the solution of \rife{vhj} is  known to decay in long time with the same rate as the heat equation (see \cite{BDL,PZ}):
\be\label{ratevhj}
\|u(t)\|_{L^\infty(\Omega)}+ \sqrt t\, \|Du(t)\|_{L^\infty(\Omega)} \leq C \, e^{-\lambda_1 t}
\ee
where $\lambda_1$ is the first eigenvalue of the Laplacian with Dirichlet conditions.

The aforementioned results strongly rely on the maximum principle and on the $L^\infty$-contr\-action property, as well as on the standard regularizing properties of the heat semigroup, including Bernstein's type estimates which are often employed to handle the decay of $Du$ and therefore of the superlinear term.

The situation is much more delicate if one considers {\it unbounded} initial data and solutions. Indeed, in the {\it unbounded} setting, the superlinearity has relevant effects and the results are no more identical 
to the standard heat equation.
A  satisfying  theory  was developed in \cite{BASW} when $u_0$ belongs to a  Lebesgue space $\elle{\si}$. In particular, the authors show that some necessary extra conditions are needed both for the existence and for the uniqueness of solutions. For the existence, it is needed that $q\leq 2$ and that $\si \geq \frac{N(q-1)}{2-q}$ (if $q<2$). For the uniqueness, some suitable class of solutions is also needed, where the authors use the gradient estimates of the heat semigroup. Those results should be compared with similar ones which hold for the superlinear problem
 \be\label{power}
 u_t -\Delta u = |u|^q
 \ee
where again some necessary restriction on the initial Lebesgue class is needed in order that  a suitable well-posed local theory be developed (see e.g. \cite{BrCa}).

The starting point of our work is the remark that the methods used in \cite{BASW} to deal with such problem rely on the well-known regularizing properties of the heat kernel, and in particular on the gradient estimates of the heat semigroup which are, essentially, a consequence of the Calderon-Zygmund $W^{2,p}$ regularity. To this respect, the above methods do not seem to extend to general operators having merely bounded measurable coefficients. For instance, the above theories do not apply even to slightly inhomogeneous variation of the reference problem \rife{vhj} as the following
\be\label{mod}
\begin{cases}
\begin{array}{ll}
\ds  u_t -\dive(A(t,x)D u) = |Du|^{q}   &\ds  \text{in }Q_T,\\
\ds  u=0  &\ds \text{on }(0,T)\times \partial \Omega,\\
\ds  u(0,x)=u_0(x) &\ds \text{in } \Omega,
\end{array}
\end{cases}
\ee
where $A(t,x) $ is a  coercive matrix with measurable and bounded coefficients.  

The goal of our paper is to give a general basic theory for equations with possibly non smooth coefficients, providing  a suitable setting for the existence and uniqueness of solutions,   showing that regularizing effects and long time decay may be proved in this more general setting as well. The main point that we wish to address is that  a satisfying theory does not need the support of the Calderon-Zygmund estimates since the subquadratic growth of the nonlinearity is enough to handle  both local and global behaviour of solutions in a suitable energy class. 

In our study, we leave aside the case of quadratic growth $q=2$ for two main reasons: first of all, this case has been already extensively studied, even in  a general nonlinear framework (see e.g. \cite{DGP,DGS1,DGS2} for existence results and global a  priori estimates), secondly this case always appears to be  special since one can rely on the Hopf-Cole transform in order to get rid of the superlinear first order term. In this way, the case $q=2$ can be reduced to the analysis of semilinear equations  and,  even the short and long time behavior should be deduced through this strategy. This is of course no longer true for the growth $q<2$, which is precisely the object of this work.
\medskip

Let us now be more precise concerning the results that we prove.
For the sake of clarity,  at first we state  our results for the model problem \rife{mod},   in order to better understand the main features of  our work.


\begin{theorem}\label{sap2}
Let $\Omega$ be a bounded domain in $\R^N, N>2$. Let $A(t,x)$ be a $N\times N$  matrix, with  bounded measurable coefficients $a_{i,j}(t,x)$, which is assumed to  be coercive (i.e. $A(t,x)\xi\cdot \xi\geq \alpha |\xi|^2$ for all $\xi\in \R^N$, for some $\alpha>0$).  Let us set
$$
\sigma:= \begin{cases} \ds \frac{N(q-1)}{2-q} & \hbox{if $\,\,\ds 2-\frac N{N+1}<q<2\,$,} \\
1 & \hbox{if $\,\,\ds 1<q< 2-\frac N{N+1}\,$.}
\end{cases}
$$
%
Then we have:
\begin{enumerate}[(i)]
\item there exists a unique weak solution $u$ of \rife{mod}, namely a  unique function $u\in L^2_{loc}((0,T); H^1_0) \cap C^0([0,T];L^\si(\Omega))$ which is a distributional solution of \rife{mod} in $Q_T$ and satisfies $u(0)=u_0$ in $\elle\sigma$ .

\item the solution $u$ satisfies the short-time estimates
\be\label{short1}
\|u(t)\|_{L^{r}(\Omega)} \le \frac{C_r}{t^{N\frac{(r-\si)}{2r\si}}}\qquad\forall\,\,t\in(0,1)
\ee
for any  $r>\si$, where $C_r$ depends on $r,N,\Omega, q, u_0$, and
\be\label{short2}
\|u(t)\|_{L^{\infty}(\Omega)} \le \frac{C_0}{t^{ \frac{N}{2\si}}}\qquad\forall\,\,t\in(0,1)
\ee
where $C_0$ depends on  $N,\Omega, q, u_0$. In addition the constants $C_r, C_0$  can be  uniformly chosen whenever $u_0$ varies in compact sets of $\elle{\si}$;

\item the solution $u$ satisfies the long-time decay
\be\label{long}
\|u(t)\|_{L^{\infty}(\Omega)}\le C_1\, e^{-\lambda_1 t}\quad\forall\,\,t\in(1,\infty)
\ee
where $C_1$  depends on  $N,\Omega, q, u_0$ and can be  uniformly chosen whenever $u_0$ varies in compact sets of $\elle{\si}$.
\end{enumerate}
\end{theorem}
 
Let us comment  below the above result in more details.

\begin{itemize}

\item When $2-\frac N{N+1}<q<2$, the Lebesgue class $\elle{\frac{N(q-1)}{2-q}}$ for the initial data is optimal in order to have the existence of a solution.  
This was already shown in \cite{BASW} for the Cauchy problem and the heat operator. In the same spirit,   we also give a counterexample for the Dirichlet problem  in \cref{optimalsec},  showing optimality of this threshold. 

\item The borderline case $q= 2-\frac N{N+1}$ is a bit special  and $L^1$-data  are not generically admitted, in this case optimal classes necessarily lead outside Lebesgue spaces. However, results can be given for $u_0\in \elle r$, whatever $r>1$, see \cref{q=}.

\item Similar results hold for $N=1$ and $N=2$, but the value of $\sigma$ should be suitably adapted according to the  Sobolev embedding for $H^1_0(\Omega)$ in dimensions $N=1$ or $N=2$, see \cref{N=}.

\item 
The class of solutions which is considered in the above statement consists, roughly speaking, of finite energy solutions at positive time, but the only global (in time) requirement is the continuity in $\elle\sigma$. It is remarkable that this class is enough to provide uniqueness. 

However, other weak formulations are also possible. When $2-\frac N{N+1}<q<2$, we  show that an equivalent formulation can be considered where, rather than requiring the $\continue\sigma$ regularity, we ask  the solutions to enjoy the global energy information:
$$
(1+|u|)^{\frac \sigma 2-1}u \in L^2(0,T; H^1_0(\Omega))
$$
and the weak formulation is extended to $[0,T)\times \Omega$, eventually using the framework of renormalized solutions if $\sigma<2$ (see \cref{def2,defrin}). We discuss those formulations  in \cref{notions} (or in \cref{datil1} for $L^1$-data) and we prove their equivalence in \cref{app}.  
\end{itemize}

Let us now spend a few words on the main ingredients of our proofs.
We stress that  the uniqueness in general  may fail for  unbounded solutions of viscous Hamilton-Jacobi equations (see e.g. \cite{BaPo,ADP}).
For operators with smooth coefficients, as in \rife{vhj} considered in \cite{BASW}, one could recover uniqueness by requiring extra-integrability on the gradient of solutions. This option, which essentially allows one to use  a linearization approach for proving uniqueness, can be justified whenever gradient estimates (or, say, the $W^{2,p}$-regularity) are achievable,  but certainly is not affordable in the context of merely bounded measurable coefficients  as we deal with. In fact, we prove the uniqueness with a refinement of the convexity argument  used in  \cite{BaPo},  similarly as it has been developed  in \cite{LePo} for the stationary case.

%
%
%
%
%
%

As far as the large time decay estimates  are concerned, we need  a  different strategy compared to what was done in \cite{BDL,PZ}. In those papers the authors used either the $L^\infty$- convergence given by viscosity solutions theory (combined with the heat kernel estimates) or the Bernstein's method for a gradient decay.

The idea of our proof is different though very simple and combines two facts which are  fundamental in this type of problems. The first one is that when data are small in $\elle\si$ then the equation behaves like the unperturbed one: in particular, for small data $u_0$ we have the standard exponential estimate in $L^\si$
$$
\frac{d}{dt} \| u(t)\|_{\elle\si}^\sigma + 4\, \lambda_1\frac{\alpha(\si-1)}{\si}  \| u(t) \|_{L^{\si}(\Omega)}^{\si}\le 0\,.
$$ 
Notice that this estimate is true for any $\si$ in the case of the heat equation (and the best decay rate follows by taking the optimal case $\sigma=2$). By contrast, in the case of problem \rife{mod}, one obtains  a similar inequality  only for the special $\si$  given by \rife{ID1},  and by requiring additionally that  the data be small in this class.
On the other hand, this property remains  true for merely subsolutions, so we can use it for  $(u-k)_+$ which can be made small for sufficiently large $k$. This also explains the kind of short-time estimates \rife{short1}-\rife{short2}, which do not depend only on the norm of the initial data, but actually rely on their equi-integrability. To this extent, even the short-time estimates are somehow different than in the case of coercive operators (see e.g. \cite{Po}).

The second ingredient, which is essential for the long-time behaviour, is that the $L^\infty$-norm is decreasing. In  this kind of problems  this is standard for classical solutions, but we show it still holds in our class of weak solutions.  

Finally, we reason by splitting our solution as the sum of a term which is small in $\elle\si$ - which decays exponentially - plus  a bounded term.  A  suitable combination of 
the short-time regularizing effect and the decreasing character of the $L^\infty$-norm will allow us to conclude the global decay \rife{long}.
\vskip0.5em

Most of the conclusions of \cref{sap2} continue to hold in a much wider generality, specifically for nonlinear coercive operators in divergence form (with just measurable bounded coefficients) and  nonlinear Hamiltonians $H(t,x,\xi)$. To be more precise, let us consider the  following Cauchy-Dirichlet problem:
\begin{equation}\label{P}\tag{P}
\begin{cases}
\begin{array}{ll}
\ds  u_t-\text{div} \left(a(t,x,u, D u) \right)=H(t,x, D u) &\ds  \text{in }Q_T,\\
\ds u=0  &\ds \text{on }(0,T)\times \partial \Omega,\\
\ds u(0,x)=u_0(x) &\ds \text{in } \Omega,
\end{array}
\end{cases}
\end{equation}
where $a(t,x,u,\xi):(0,T)\times \Omega\times \mathbb{R}\times \mathbb{R}^N\to  \mathbb{R}^N$ and $H(t,x,\xi):(0,T)\times\Omega\times \mathbb{R}^N\to \mathbb{R}$ are Caratheodory functions (i.e. measurable with respect to $(t,x)$ and continuous in $(u,\xi)$)
satisfying the following coercivity and growth conditions:
\begin{equation}\label{A1}\tag{A1}
\exists \,\alpha>0:\quad
\alpha|\xi|^2\le a(t,x,u,\xi)\cdot\xi,
\end{equation}
\begin{equation}\label{A2}\tag{A2}
\exists \,\lambda>0:\quad |a(t,x,u,\xi)|\le \lambda[|u|+|\xi|+h(t,x)],\quad h\in L^{2}(Q_T),
\end{equation}
and
\begin{equation}\label{H}\tag{H}
\exists \,\,\gamma>0\,\,\text{s.t. }\,\,
|H(t,x,\xi)|\le \gamma |\xi|^q\quad \text{with}\quad 1< q<2
\end{equation}
for almost every $(t,x)\in Q_T$ and for every $\xi\in \mathbb{R}^N$.

As already explained above, according  to the value of $q$,  the initial datum $u_0$ is taken in the following Lebesgue spaces, at least if $N>2$ (for $N=1,2$, see \cref{N=}):
\begin{equation}\label{ID1}\tag{ID1}
u_0\in L^{\si}(\Omega)\quad\text{with}\quad \si=\frac{N(q-1)}{2-q}\,,\quad\text{if} \quad 2-\frac{N}{N+1}<q<2,
\end{equation}
\begin{equation}\label{ID2}\tag{ID2}
u_0\in L^1(\Omega)\,,\quad\text{if} \,\,\,1<q<2-\frac{N}{N+1}.
\end{equation}
Eventually, by a weak solution of \rife{P} we mean a  function $u\in L^2_{loc}((0,T);H^1_0(\Omega))$ which is a distributional solution in $Q_T$ and belongs to $\continue\sigma$, if $2-\frac{N}{N+1}<q<2$ or, respectively, to $\continue1$ if  $1<q<2-\frac{N}{N+1}$. Other  formulations, equivalent to the previous one but yet interesting (including notions of renormalized solutions), are discussed in detail in the next sections. Let us stress that  the growth assumption \rife{A2} is only assumed in order to give a standard meaning to those formulations (e.g. $\text{div} \left(a(t,x,u, D u) \right)$ has a  usual distributional meaning whenever $u\in L^2_{loc}((0,T);H^1_0(\Omega))$). However, \rife{A2} does not play any significant role in all the estimates obtained, which only rely on the coercivity of the second order part and on the growth of the first order terms.

Similar results as stated in \cref{sap2} can be obtained in the general case.  

\begin{theorem} Let $N>2$, and assume that \rife{A1}-\rife{A2} and \rife{H} hold true, with $u_0$ satisfying either \rife{ID1} or \rife{ID2}. 

Then,  there exists a weak solution of \rife{P}. In addition, any weak solution satisfies the same estimates \rife{short1}--\rife{long} as in \cref{sap2}, where now the constants also depend on $\alpha$ and $\gamma$ (appearing in \rife{A1}, \rife{H} respectively).
\end{theorem}

The above theorem is only a rough summary of the main results that we prove in this paper, but more detailed, and sometimes more general,  statements, are given  in the next sections. 

For example, we also provide with  uniqueness results in the class of weak solutions, at least when the second order operator is linear, namely    if $a(t,x,u,\xi)= A(t,x)\cdot Du$ for some coercive bounded matrix $A$. In that case we give  sufficient conditions   ensuring uniqueness  of the weak solution. Simple reference cases are  when $\xi\mapsto H(t,x,\xi)$ is a convex function (as in \rife{mod}), or even if it is  a $C^2$ function which is just convex at infinity.
\vskip1em

Let us conclude by explaining how the paper is organized. In \cref{finen} we discuss properties of solutions, and specifically  a priori estimates, for the case  when $2-\frac{N}{N+1}<q<2$. In \cref{deca} we prove the regularizing effects, as well as the short and long time estimates.  \cref{datil1} is devoted to the range $1<q<2-\frac{N}{N+1}$, in which   one can afford $L^1$-data (hence the a priori-estimates take a slightly different form). We briefly discuss the special cases $N=1$ and $N=2$ in \cref{N=}. Finally, \cref{uni} contains the uniqueness results, while \cref{optimalsec} shows the optimality of the class considered for the initial data.
\vskip1em

For the interested reader, let us stress that existence results for more general operators (including the $p$-Laplacian) and general right-hand sides can be found in \cite{M}. The present study of short and long time decay  can also be extended to the $p$-Laplace operator with similar techniques, but  a similar  development would have resulted in a too long exposition.

\subsection*{Notation}
We will represent the constant due to the Sobolev's (Poincar\'{e}'s) inequality by $c_S$ ($c_P$) while $c$ will stand for a positive constant which may vary from line to line during the proofs. Moreover, we define some auxiliary function  which will be used in the next Sections:
\be\label{gktk}
G_k(v)=(|v|-k)_+\text{sign}(v),\qquad T_k(v)=v-G_k(v)=\max\{-k,\min\{k,v\}\}.
\ee

\section{The  case $2-\frac{N}{N+1}< q<2$. Estimates in $\elle\sigma$.}\label{finen} 

In this Section we assume that, accordingly to the growth of the function $H(t,x,Du)$ (see \rife{H}), we have $2-\frac{N}{N+1}< q<2$ and $u_0$ satisfies \rife{ID1}. The value $\sigma$ is defined here as 
$$
\sigma= \frac{N(q-1)}{2-q}\,.
$$

\subsection{Notions of solution}\label{notions}

Different notions of solutions could be suggested for problem   \eqref{P}.  We will currently work with the following one, in which we require  solutions to be in the standard energy space for $t>0$, while the only {\it global} information is that $u\in \continue {\si}$.

\begin{definition}\label{def1} A  function $u$ is a solution to \rife{P} if $u\in \continue {\si} \cap \paraccaloc$, $u(0)= u_0$ and $u$ satisfies
\begin{equation*}
  \iint_{Q_T} - u\varphi_t+ a(t,x,u,D u)\cdot  D \varphi\,dx\,dt =  \iint_{Q_T} H(t,x, D u)\varphi\,dx\,dt
\end{equation*}
for every $\varphi\in C^\infty_c(Q_T)$.
\end{definition}

Let us soon point out that other notions of solutions are equally interesting. In particular, it could be meaningful to require some energy  to be finite in the whole cylinder $Q_T$, while no explicit time continuity (nor boundedness) being  a  priori required. In this case, it turns out that the natural energy space consists of functions $u$ such that 
$(1+|u|)^{\frac{\si}{2}-1}u \in L^2(0,T;H_0^1(\Omega))$. Therefore, a  second possible notion of solution is the following one.

\begin{definition}\label{def2}
A function $u$ is a solution of \eqref{P} if
\begin{equation}\label{beta}\tag{RC}
(1+|u|)^{\frac{\si}{2}-1}u \in L^2(0,T;H_0^1(\Omega))\quad\text{for}\quad\si=\frac{N(q-1)}{2-q}\,,
\end{equation}
if $H(t,x, D u)\in L^1(Q_T)$ and $u$ satisfies the weak formulation
\[
-\int_{\Omega}u_0\varphi(0)\,dx+\iint_{Q_T}[-\varphi_t u +a(t,x,u, D u)\cdot  D \varphi]\,dx\,dt =\iint_{Q_T}H(t,x, D u)\varphi \,dx\,dt
\]
for every   $\vp\in C^\infty_c([0,T)\times \Omega)$.
\end{definition}

\begin{remark}\label{rmkct}
We notice that, if $\sigma\geq 2$, condition \rife{beta} implies $ u\in L^2(0,T; H_0^1(\Omega)) $. In this case, since $q<2$, we already have from \rife{H} that $H(t,x,Du)\in  L^1(Q_T)$. Moreover, if \rife{A2} holds, we deduce from the distributional formulation    that
\[
u_t=\text{div}(a(t,x,u, D u))+H(t,x, D u)\in L^2(0,T;H^{-1}(\Omega))+L^1(Q_T).
\]
Thus, thanks to \cite[Theorem $1.1$]{P1}, we can say that $u$ belongs, at least, to $C^0([0,T];L^1(\Omega))$.
\end{remark}

It is interesting to notice that the above two notions coincide, at least if $\si\geq 2$, in which case \rife{beta} implies that solutions  belong to $L^2(0,T;H_0^1(\Omega))$ (this range of $\si$ corresponds to $2-\frac{N}{N+2}\le q<2$). This is the content of the next result. 

\begin{proposition}\label{defequiv} Assume that \rife{A1}-\rife{A2} and \rife{H} hold and that $2-\frac{N}{N+2}\le q<2$ (i.e. $\si\geq 2$). Then  \cref{def1} and \cref{def2} are equivalent. 
\end{proposition}

The proof of \cref{defequiv} is a bit technical and we postpone it to the \cref{app}, although several ingredients will be already contained in the next sections.

In the case that $1 < \si <2 $,  which corresponds to $2-\frac N{N+1} < q < 2-\frac{N}{N+2}$, the regularity class \rife{beta} does not allow us to consider finite energy solutions any more. However, a similar equivalent formulation holds if we complement \cref{def2} by requiring that  a standard chain rule argument is allowed. This usually goes under the name of {\it renormalized formulation}, in which the truncations of the solution - if not the solution itself - are required to be in the energy space, see \cite{BDGM,BlMu,DMOP,Mu}.

We recall from \cite{BBGGPV} that the set of functions $\mathcal{T}^{1,2}_0(Q_T)$ collects all the measurable functions $u:Q_T\to \mathbb{R}$ almost everywhere finite and such that the truncated functions $T_k(u)$ belong to $ L^2(0,T;H_0^1(\Omega))$ for all $k>0$. The generalized gradient of a function $u$ in $\mathcal{T}^{1,2}_0(Q_T)$ is defined as 
\[
D T_k(u)=D u \chi_{\{|u|< k\}}.
\]

 \begin{definition}\label{defrin}
We say that a function $u\in \mathcal{T}^{1,2}_0(Q_T)\cap \parelle1$ is a renormalized solution of \eqref{P} if $u$ satisfies \rife{beta}, 
\begin{subequations}
\def\@currentlabel{RS}
\makeatother
\label{RS}
\renewcommand{\theequation}{RS\arabic{equation}}
\begin{equation}
 H(t,x, D u)\in L^1(Q_T),\label{sr1} 
\end{equation}
and if the following renormalized equation holds:
\begin{equation}\label{sr2}
\begin{array}{c}
\ds
-\integrale S(u_0)\vp (0,x)\,dx+\iint_{Q_T}-S(u)\vp_t+a(t,x,u,D u)\cdot D (S'(u)\vp)\,dx\,dt\\
[3mm]\ds
=\iint_{Q_T}H(t,x,D u)S'(u)\vp\,dx\,dt
\end{array}
\end{equation}
\end{subequations}
for every $S\in W^{2,\infty}(\mathbb{R})$ such that $S'(\cdot)$ has compact support and any  $\vp\in C^\infty(\overline Q_T)$ such that  $\vp(T,x)=0$ and $S'(u)\vp \in L^2(0,T;H_0^1(\Omega))$ (i.e., the product $S'(u)\vp$ is equal to zero on the  boundary $(0,T)\times \partial \Omega$).
\end{definition}


\begin{remark}
We observe that, differently than in the usual notions of renormalized solution (\cite{BlMu,DMOP}),  no asymptotic condition on the energy term appears in \cref{defrin}. In fact, it is the regularity assumption \rife{beta}  which, in particular, ensures that
\[
\lim_{n\to\infty}\frac{1}{n}\iint_{\{n<|u|<2n\}}a(t,x,u,D u)\cdot D u\,dx\,dt\underset{n\to\infty}{\longrightarrow}0.
\]
As is well known, the above asymptotic condition is  a crucial information for the renormalized formulation and, in particular, it implies that  a standard distributional formulation is also verified.
\end{remark}

We show in the next result that the above renormalized formulation is also equivalent to \cref{def1}. This holds true for every value of $\si>1$ although the interesting case is when $\si<2$, since if $\si\geq 2$ \cref{defrin} is equivalent  to \cref{def2} and \cref{defequiv} applies.

\begin{proposition}\label{equivren}  Assume that \rife{A1}-\rife{A2} and \rife{H} hold  and that $2-\frac{N}{N+1}< q<2$. Then  \cref{def1,defrin} are equivalent.
\end{proposition}

We also postpone the proof of this result to the \cref{app}. Let us recall that the existence of at least one solution satisfying \rife{def2} (if $2-\frac{N}{N+2}\leq  q<2$) or \rife{defrin} (if $2-\frac{N}{N+1}< q<2-\frac{N}{N+2}$) is proved in \cite{M}.
 
%
%

Eventually, similar notions of {\it subsolution} (or \emph{supersolution}) could be defined with the obvious variations of  the above definitions. With   similar proofs as before, it follows that for a function $u\in \continue{\si}$ the different notions of subsolution coincide.

\subsection{A priori estimates and contraction principles.}\label{contraction}

In this section we give evidence to the crucial role played by the space $\elle{\si}$ in the study of \rife{P}. We start from a fundamental contraction principle which holds in this space.

In the following, we simply refer to  a solution of \rife{P} in the sense of \cref{def1}. 
Let us first observe  that a standard chain rule and integration by parts  formula hold for those solutions, namely we have the following

\begin{proposition}\label{propint}
Let $u$ be a solution of \eqref{P}. Then
\begin{equation*}
\begin{array}{c}
\ds
\int_{\Omega}\Psi(u(t_2))\,dx+\int_{t_1}^{t_2}\into a(s,x,u, D u)\cdot  D u\,\psi'(u)\,dx\,ds\\
[3mm]\ds
=\int_{t_1}^{t_2}\into H(s,x, D u)\psi(u) \,dx\,ds+\int_{\Omega}\Psi(u(t_1))\,dx \qquad \text{for every }\quad t_1, t_2\in(0,T),
\end{array}
\end{equation*}
for every $\psi\in W^{1,\infty}(\mathbb{R})$ such that $\psi(0)=0$, where $\Psi(v)=\int_0^v \psi (w)\,dw$.
\end{proposition}
\proof
Whenever $\psi'(\cdot)$ has compact support, this Proposition is contained e.g. in \cite[Lemma 7.1]{DP} (or in \cite[Corollary $4.7$ and Remark $4.2$]{PPP}). The conclusion  immediately extends to $\psi\in W^{1,\infty}(\mathbb{R})$ because $u\in \continue {\si}\cap \paraccaloc $.  \endproof

Now we prove the main  step upon which all a priori estimates will rely. Recall that the function $G_k(s\cdot)$ is defined in \rife{gktk}.

\begin{lemma}\label{lemGk}
Assume \eqref{A1},   \eqref{H} with $2-\frac{N}{N+1}< q<2$ and \eqref{ID1}. Let $u$ be a solution of \eqref{P}. Then,  there exists a positive value $\delta_0$, only depending on $\alpha$, $c_S$, $q$ and $\gamma$, such that, for every $k>0$ and for every $\delta<\delta_0$ satisfying
\[
\|G_k(u_0)\|_{L^{\si}(\Omega)}^{\si}<\delta
\]
we have
\begin{equation}\label{lemgk}
\|G_k(u(t))\|_{L^{\si}(\Omega)}^{\si}<\delta\quad \forall\, t\in[0,T].
\end{equation}
\end{lemma}

\begin{remark}\label{precise}
According to the proof below, the value of $\delta_0$ can be explicitly given by the equality $2^q\frac{c_S \gamma}{\si^q} {\delta_0}^{\frac{2-q}{N}}=4\alpha\frac{\si-1}{\si^2}$.
\end{remark}

\proof

We divide the proof in two steps. \\

\noindent
{\it Step 1} \\
\noindent
Here we show that any solution $u$ according to \cref{def1}  satisfies the properties that
\begin{equation}\label{reggku}
|G_k(u)|^{\frac{\si}{2}-1}G_k(u)\in L^2(0,T;H_0^1(\Omega)),
\end{equation}
\begin{equation}\label{pro2}
(1+|u|)^{\frac{\si}{2}-1}u \in L^2(0,T;H_0^1(\Omega))
\end{equation}
and the following inequality holds for every $k>0$:
\begin{equation}\label{lim}
\begin{array}{c}
\ds
\frac{1}{\si}\int_{\Omega}|G_k (u(t))|^{\si}\,dx+
4\alpha\frac{(\si-1)}{\si^2}
\iint_{Q_t} | D [|G_k(u)|^{\frac{\si}{2}-1}G_k(u)] |^2 \,dx\,ds
\\
[3mm]\ds
\le\frac{2^q\gamma  c_S}{\si^q}\int_0^t\biggl(\int_{\Omega} | D [|G_k(u)|^{\frac{\si}{2}-1}G_k(u)]|^2 \,dx\biggr)\biggl(\int_{\Omega}  |G_k(u(s))|^{\si}\,dx\biggr)^{\frac{2-q}{N}}\,ds\\
 [3mm]\ds
 +\frac{1}{\si}\int_{\Omega}|G_k(u_0)|^{\si}\,dx.
\end{array}
\end{equation}
for every $t\in (0,T]$.  
\\

In order to prove \rife{lim}, we start by considering the case that $\si\geq 2$; then we use $\psi_n(u)= G_k(u)|T_n(G_k(u))|^{\si-2} $ in \cref{propint}. By assumptions \eqref{A1},  \eqref{H}, we get
\begin{equation*} 
\begin{array}{c}
\ds
\into \Psi_n(u(t_2))dx  + \alpha \int_{t_1}^{t_2}\into |DG_k(u)|^2\,\psi_n'(u) dx\,dt
\\
[3mm]\ds\leq \into \Psi_n(u(t_1))dx
+ \gamma \int_{t_1}^{t_2}\into |DG_k(u)|^q\,|T_n(G_k(u))|^{\si-2} |G_k(u)|\,dx\,dt
\end{array}
\end{equation*}
for every $t_1,t_2\in (0,T)$, where $\Psi_n(s)=\int_0^s \psi_n(\xi)d\xi$. 

We estimate the last integral by H\"older's inequality with indices $\left(\frac{2}{q}, \frac{2^*}{2-q},\frac{N}{2-q}\right)$ and we get
\be\label{psiu}
\begin{array}{c} 
\ds
\into \Psi_n(u(t_2))dx   + \alpha \int_{t_1}^{t_2}\into |DG_k(u)|^2\,\psi_n'(u) dx\,dt\\
[3mm]\ds
\leq \into \Psi_n(u(t_1))dx 
+ \gamma \int_{t_1}^{t_2}\biggl[\left(\into |DG_k(u)|^2\,|T_n(G_k(u))|^{\si-2}dx\right)^{\frac q2}\times
\\
[3mm]\ds
\times \left(\into |T_n(G_k(u))|^{\frac {(\si-2)N}{N-2}}|G_k(u)|^{2^*}dx\right)^{\frac{2-q}{2^*}} \left(\into |G_k(u)|^{\si}\,dx\right)^{\frac{2-q}N}\biggr]dt\,,
\end{array}
\ee 
where we used the exact value of $\si$.

Let us now justify that we can take the limit as $n\to \infty$ in \rife{psiu}. Indeed, 
since $\Psi_n(u)\leq c \, |u|^{\si}$, the terms at $t_1$ and $t_2$ are bounded uniformly with respect to $n$.  Using that  $\psi_n'(u)\geq |T_n(G_k(u))|^{\si-2}$ we deduce that 
\be\label{tn}
\begin{array}{c}
\ds \alpha\int_{t_1}^{t_2}\into |DG_k(u)|^2\,|T_n(G_k(u))|^{\si-2}dx\,dt\\
[3mm]\ds
\leq c
+ c \sup\limits_{[t_1,t_2]} \|G_k(u(t))\|_{\elle{\si}}^{\frac{\si(2-q)}N} \times
\\
[3mm]\ds \times 
\int_{t_1}^{t_2}\left(\into |DG_k(u)|^2\,|T_n(G_k(u))|^{\si-2}dx\right)^{\frac q2} \left(\into |T_n(G_k(u))|^{\frac{(\si-2)N}{N-2}}|G_k(u)|^{2^*}dx\right)^{\frac{2-q}{2^*}} dt.
\end{array}
\ee
Finally, since $\int_0^s |T_n(\xi)|^{\frac{\si}{2}-1}d\xi \geq c\, |s\, T_n(s)^{\frac{\si}{2}-1} |$ for some constant independent of $n$, we have
$$
\into |T_n(G_k(u))|^{ \frac{(\si-2)N}{N-2}} |G_k(u)|^{2^*}\,dx \leq c \into \left | \int_0^{G_k(u)} |T_n(\xi)|^{\frac{\si}{2}-1}d\xi\right|^{2^*}
$$
so using Sobolev inequality we conclude from \rife{tn} that
\begin{equation*} 
\begin{array}{c}
\ds \int_{t_1}^{t_2}\into |DG_k(u)|^2\,|T_n(G_k(u))|^{\si-2}dx\,dt 
\\
[3mm]\ds
\leq c+ c \sup\limits_{[t_1,t_2]} \|G_k(u(t))\|_{\elle{\si}}^{\frac{\si(2-q)}N}  \left(\int_{t_1}^{t_2}\into |DG_k(u)|^2\,|T_n(G_k(u))|^{\si-2}dx\,dt\right)\,.
\end{array}
\end{equation*}
By Dini's convergence Theorem, we have that $\|G_k(u(t))\|_{\elle{\si}}$ converges to zero as $k\to \infty$, uniformly for $t\in (0,T)$. Therefore, there exists $k^*$ such that for $k\geq k^*$ the right-hand side can be absorbed into the left-hand side and  a uniform estimate follows:
\be\label{reggkun}
\int_{t_1}^{t_2}\into |DG_k(u)|^2\,|T_n(G_k(u))|^{\si-2}dx\,dt\leq c\,.
\ee
By letting $n\to \infty$, and since $(t_1,t_2)$ is arbitrary,  we conclude that \eqref{reggku} holds
for $k\geq k^*$.

We go back now to the inequality   \eqref{psiu} where we can take the limit as  $n\to \infty$. 
Since $\psi_n(s)\to |G_k(s)|^{\si-2}G_k(s)$, once we take the limit we obtain
\besac 
\ds
\frac{1}{\si}\int_{\Omega}|G_k (u(t_2))|^{\si}\,dx+
\alpha (\si-1)\int_{t_1}^{t_2}\into |DG_k(u)|^2\,|G_k(u))|^{\si-2}dx\,dt\\
[3mm]\ds
\leq \frac{1}{\si}\int_{\Omega}|G_k (u(t_1))|^{\si}\,dx
\\
[3mm]\ds
 + \gamma \int_{t_1}^{t_2}\left(\into |DG_k(u)|^2\,|G_k(u)|^{\si-2}dx\right)^{\frac q2} \left(\into |G_k(u)|^{\si \frac{2^*}{2}}dx\right)^{\frac{2-q}{2^*}} \left(\into |G_k(u)|^{\si}\,dx\right)^{\frac{2-q}N}dt\,.
\eesac
Using Sobolev inequality in the right-hand side, and letting $t_1\to 0$ (which is licit since $u\in \continue{\si}$), we conclude that \rife{lim} holds true.
Moreover, we notice that, once we know that \rife{reggku} holds for {\it some} $k>0$, this already implies that $|G_k(u)|^\frac{\si}{2} \in L^2(0,T; \elle{2^*})$ for {\it any} $k>0$, hence from \rife{tn} using $q<2$ one easily obtains that estimate \rife{reggkun} is still true for any possible value of $k$. Therefore, as before one deduces that \rife{reggku} holds true and one gets at inequality \rife{lim} for every $k>0$. 

Let us now consider the case that $\si\in (1,2)$. Here we take $\psi_n(u)=(\si-1)\int_0^{G_k(u)} (\vep_n+|\xi|)^{\si-2}d\xi$ where $\vep_n=\frac1n $. 
Since $|\psi_n(u)|\leq (\vep_n+|G_k(u)|)^{\si-1}$ in this case \rife{psiu} can be replaced by 
 \besac
\ds\into \Psi_n(u(t_2))dx   + \alpha \int_{t_1}^{t_2}\into |DG_k(u)|^2\,\psi_n'(u) dx\,dt\\
[3mm]\ds
\leq \into \Psi_n(u(t_1))dx
 + \gamma \int_{t_1}^{t_2}\biggl[\left(\into |DG_k(u)|^2\,(\vep_n+|G_k(u)|)^{\si-2}dx\right)^{\frac q2}  \left(\into (\vep_n+|G_k(u)|)^{ \si\,\frac{2^*}{2}} dx\right)^{\frac{2-q}{2^*}}\times \\
  [3mm]\ds\times \left(\into (\vep_n+|G_k(u)|)^{\si}\,dx\right)^{\frac{2-q}N}\biggr]dt\,.
\eesac
This inequality in particular yields
 \besac
\ds \int_{t_1}^{t_2}\into |DG_k(u)|^2\,(\vep_n+|G_k(u)|)^{\si-2} dx\,dt\leq c
\\
[3mm]\ds + c\int_{t_1}^{t_2} \left(\into (\vep_n+|G_k(u)|)^{ \frac{2^*\si}{2}} dx\right)^{\frac{2}{2^*}} \left(\into (\vep_n+|G_k(u)|)^{\si}\,dx\right)^{\frac{2}N}dt\,,
\eesac
and since
$$
(\vep_n+|G_k(u)|)^{\si\,\frac{2^*}{2}} \leq c \left[\left(\int_0^{G_k(u)} (\vep_n+|\xi|)^{ (\frac{\si}{2}-1)}\,d\xi\right)^{2^*}+ \vep_n^{ \frac{2^*\si}{2}}\right]
$$
by Sobolev's inequality  we get
$$
\left\{1-c \left[ \vep_n^{\frac{2\si}N}+ \| \sup\limits_{[t_1,t_2]} \|G_k(u(t))\|_{\elle{\si}}^{\frac{2\si}N}\right]\right\}
\int_{t_1}^{t_2}\into |DG_k(u)|^2\,(\vep_n+|G_k(u)|)^{\si-2} \,dx\,dt\leq c\,.
$$
Thus we conclude again an estimate, for sufficiently large $k$,  which is uniform with respect to $n$ and which allows us to deduce that \rife{reggku} holds true once we take the limit as $n\to \infty$. In addition, since again we have  $\psi_n(s)\to |G_k(s)|^{\si-2}G_k(s)$, we also recover inequality \rife{lim} for the case $\si<2$. Eventually, with the same argument as before both \rife{reggku} and \rife{lim} are shown to hold for {\it every} $k>0$. 

To complete the claim of this Step 1, we only need to show that $(1+|u|)^{\frac{\si}{2}-1}u \in L^2(0,T;H_0^1(\Omega))$, and this is now concerned only with the bounded values of $u$, since we already know that $ |G_k(u)|^{\frac{\si}{2}-1}G_k(u) \in L^2(0,T;H_0^1(\Omega))$.
Thus we  use $T_{k}(u)$ as test function obtaining
\beac\label{tkk}
\ds\int_{\Omega}\Theta_{k}(u(t_2))\,dx+  \alpha\int_{t_1}^{t_2}\into| D T_{k}(u)|^2\,dx\,dt\\
[3mm]\ds
\le  \int_{\Omega}\Theta_{k}(u(t_1))\,dx +\gamma \int_{t_1}^{t_2}\into | D u|^q|T_{k}(u)|\,dx\,dt ,
\eeac
where $\Theta_{k}(s)=\int_0^s T_{k}(z)\,dz$.\\
Using the decomposition $u=G_{k}(u)+ T_{k}(u)$ we estimate the right-hand side as
\besac
\ds
\int_{\Omega}\Theta_{k}(u(t_1))\,dx +\gamma \int_{t_1}^{t_2}\into | D u|^q|T_{k}(u)|\,dx\,dt \\
[3mm]\ds 
\leq k  \into |u(t_1)|dx + \gamma\, k  \int_{t_1}^{t_2}\into | D u|^q \,dx\,dt
\\
[3mm]\ds \leq k^2 |\Omega| + k\|G_{k}(u(t_1))\|_{\elle{\si}} |\Omega|^{1-\frac1{\si}} 
 + \frac\alpha 2   \int_{t_1}^{t_2}\into | D T_{k}(u)|^2 \,dx\,dt+ C  k^{\frac 2{2-q}} \\
 [3mm]\ds +  \ga k \int_{t_1}^{t_2}\into | D G_{k}(u)|^q \,dx\,dt,
\eesac
for some constant $C= C(\alpha, \gamma, T, |\Omega|)$. Therefore \rife{tkk} implies
\beac\label{tkk2}
 \ds \frac\alpha2\int_{t_1}^{t_2}\into| D T_{k}(u)|^2\,dx\,dt\\
 [3mm]\ds\leq k^2 |\Omega| + k\|G_{k}(u(t_1))\|_{\elle{\si}} |\Omega|^{1-\frac1{\si}} 
+   C  k^{\frac 2{2-q}}  + \ga k \int_{t_1}^{t_2}\into | D G_{k}(u)|^q \,dx\,dt
\eeac
By choosing $k= 2k^*$, if $\si\geq2$ we have 
$$
| D G_{k}(u)|^q \leq \left(\frac{G_{k^*}(u)}{k^*}\right)^{(\frac{\si}{2}-1)q}   | D G_{k^*}(u)|^q
\leq 1+ \frac4{\si^2\, (k^*)^{\sigma-2}} | D |G_{k^*}(u)|^\frac{\si}{2} |^2\,.
$$
Hence we deduce from \rife{tkk2} that
\be\label{tk*}
\hspace*{-3mm}\int_{t_1}^{t_2}\into| D T_{2k^*}(u)|^2\,dx\,dt\leq C\left(k^*, \|u\|_{\limitate{\si}},  \| |G_{k^*}(u)|^{\frac{\si}{2}-1}G_{k^*}(u)\|_{L^2((0,T);H^1_0(\Omega))}\right).
\ee
If rather $\si<2$, we estimate
$$
\into | D G_{k}(u)|^q \,dx \leq \frac q2 \into \frac{| D [|G_{k}(u)|^{\frac{\si}{2}-1}G_{k}(u)] |^2}{\si^2/4} \,dx + \left(1-\frac q2\right)
\into |G_{k}(u)|^{\frac{(2-\si)q}{2-q}} \,dx\,.
$$
But we already know that, for $k\geq k^*$,  $|G_k(u)|^{\frac{\si}{2}}\in \limitate2\cap L^2((0,T);H^1_0(\Omega))$, hence  $|G_k(u)|^{\frac{\si}{2}} \in \parelle{2\frac{N+2}N}$ by Gagliardo-Nirenberg interpolation (see e.g. \cite{Di}); since $\frac{(2-\si)q}{2-q}<\si\frac{(N+2)}N$, the last integral is also estimated by the norm of $|G_{k}(u)|^{\frac{\si}{2}-1}G_{k}(u)$ in $\limitate2\cap L^2((0,T);H^1_0(\Omega))$. Therefore, \rife{tk*} also holds for $\si<2$. 

Overall, combining estimate \rife{reggku} and \rife{tk*} we can conclude that \eqref{pro2} holds.\\

\noindent
{\it Step 2} \\
\noindent  We now define the value $\delta_0$ so that 
$$
\frac{\gamma 2^q c_S}{\si^q} {\delta_0}^{\frac{2-q}{N}}=4\alpha\frac{\si-1}{\si^2}.
$$
Let $\de<\de_0$ and take some  $k$ such that  
\[
\|G_k(u_0)\|_{L^{\si}(\Omega)}^{\si}< \de \,.
\]
Since $u\in \continue \si$ the value  
\begin{equation*}
T^*:=\sup\{t\in [0,T]: \,\|G_k(u(s))\|_{L^{\si}(\Omega)}^{\si} < \delta  \,\quad\forall\, s\le t  \}
\end{equation*}
is well posed and we have  that $T^* >0$.\\
We claim   that $T=T^*$. Indeed, if $t\le T^*$ in \eqref{lim}, then the definition of $\delta_0$ implies that
\begin{equation}\label{datot}
\begin{array}{c}
\ds
\frac1\sigma\int_{\Omega}|G_k(u(t))|^{\si}\,dx+ \frac{\gamma 2^q c_S}{\si^q} \left[{\delta_0}^{\frac{2-q}{N}}-\de^{\frac{2-q}{N}}\right]
\int_0^t\int_{\Omega}  | D [|G_{k}(u)|^{\frac{\si}{2}-1}G_{k}(u)] |^2  \,dx\,ds\\
[3mm]\ds
\le \frac1\sigma \int_{\Omega}|G_k(u_0)|^{\si}\,dx.
\end{array}
\end{equation}
Suppose, by contradiction, that $T^*<T$ and let $t=T^*$ in \eqref{datot}, so we get
\[
\int_{\Omega}|G_k(u(T^*))|^{\si}\,dx\le\int_{\Omega}|G_k(u_0)|^{\si}\,dx<\delta.
\]
Since $u\in C^0([0,T];L^{\si}(\Omega))$, if $T^*<T$ this last inequality leads to a contradiction with the definition of $T^*$. Hence we conclude that $T^*=T$.
\qed

\vskip1em

The above contraction principle has two relevant consequences. The first one is the following  a \emph{priori estimate}.

\begin{theorem}\label{sap}
Assume \eqref{A1},  \eqref{H} with $2-\frac{N}{N+1}< q<2$ and \eqref{ID1}. Let $u$ be a   solution of \eqref{P}. Then 
the following estimate holds:
\be\label{aprio}
\sup_{t\in [0,T]}\|u(t)\|_{L^{\si}(\Omega)}+\| D[(1+|u|)^{\frac{\si}{2}-1}u]\|_{L^2(Q_T)}\le M
\ee
where the constant $M$ depends on $u_0$, $T$, $|\Omega|$, $q$, $\gamma$, $N$, $\alpha$ and remains bounded when $u_0$ varies in sets which are bounded and equi-integrable in $L^{\si}(\Omega)$. 
\end{theorem}

\proof Let $\de_0$ be given as in \cref{lemGk}, and fix some $\de<\de_0$ (e.g. take $\de= \frac {\de_0}2$) and let $k_0$ be such that $ \|G_k(u_0)\|_{L^{\si}(\Omega)}^\si< \de$ for $k\geq k_0$. As a consequence of \rife{datot} in \cref{lemGk}, for any $k\geq k_0$ we estimate
$\|G_k(u)\|_{\limitate{\si}}$ and $\| D [|G_k(u)|^{\frac{\si}{2}-1}G_k(u)]\|_{L^2(Q_T)}$ only in terms of $q$, $N, \gamma, \alpha$. Next, from \rife{tkk2} in \cref{lemGk} we estimate $\| D T_k(u)\|_{\parelle2}$ in terms of the above constants and of $T, |\Omega|$and $k_0$. This latter value depends on $u_0$ and remains bounded for a  family of initial data  $u_0$ which are equi-integrable in $\elle{\si}$. Finally, it is  a simple exercise to combine these two estimates in order to obtain \rife{aprio}.
\qed

\vskip1em
A second consequence of \cref{lemGk} is the \emph{contraction property} in $\elle\infty$.

\begin{corollary}\label{corinfty}
Assume \eqref{A1},  \eqref{H} with $2-\frac{N}{N+1}< q<2$ and let the initial datum $u_0$ belong  to $L^{\infty}(\Omega)$. If $u$ is a solution of \eqref{P}, then we have that $u$ belongs to $L^{\infty}(Q_T)$ and 
\[
\|u(t)\|_{L^{\infty}(\Omega)}\le \|u_0\|_{L^{\infty}(\Omega)}\quad \forall\, t\in[0,T].
\]
\end{corollary}
\proof
Let $k= \|u_0\|_{L^{\infty}(\Omega)}$. Then \eqref{lemgk} holds with $\delta$ arbitrary small and   the assertion follows by letting $\delta$ go to zero.
\endproof

\begin{remark}\label{rmksubsol}
We notice that \cref{lemGk} and   \cref{sap} actually hold  for nonnegative subsolutions of
\begin{equation}\tag{$P_{z}$}\label{pz}
\begin{cases}
\begin{array}{ll}
\ds z_t -\text{div}(a(t,x,z, D z))=\gamma | D z|^q  &\ds \text{in}\ Q_T,\\
\ds z=0   &\ds \text{on}\ (0,T)\times \partial\Omega,\\
\ds z(0,x)=z_0(x)  & \ds \text{in}\  \Omega.
\end{array}
\end{cases}
\end{equation}
In fact, if the Hamiltonian function $H(t,x,\xi)$ has a sign, it would be more convenient to work with just nonnegative  subsolutions. Assume for instance that $u H(x,u, D u)\geq 0$; then by writing 
$u$ as the difference between its positive and negative parts, i.e. $u=u^+-u^-$, it is not difficult to see that 
\begin{equation*}
u_t^+ -\text{div}(a(t,x,u^+,  D u^+))\le H(t,x, D u ) \chi_{\{u > 0\}}  \leq \gamma | D u^+|^q
\end{equation*}
while 
\begin{equation}\label{u-}
u_t^- -\text{div}(a(t,x,u^-,  D u^-))\le - H(t,x, D u ) \chi_{\{u \leq  0\}}  \le 0.
\end{equation}
Therefore the problem related to $u^-$ turns out to be a  standard coercive problem, whereas all the difficulties of the superlinear term are observed in the behavior of $u^+$. In particular, in this case all necessary conditions and any other effect of superlinearity should be discussed in terms of $u_0^+$ only.
\end{remark}

\subsection{The borderline case $q=2-\frac{N}{N+1}$}\label{q=}
This particular value of $q$ implies that $\sigma=1$. However, $L^1$-data are not admissible for a general solvability in this case (see also \cite{BASW,M}). Therefore, in the class of Lebesgue spaces we should 
 take into account initial data belonging to $L^{1+\omega}(\Omega)$ for some $\omega\in(0,1)$. This stronger assumption allows us to apply directly what we obtained before: indeed, given $\omega>0$ such that $u_0\in L^{1+\omega}(\Omega)$, then it is always possible to embed our nonlinearity into a  $q$-growth for some $q=q_\omega$ such that
 $1+\omega= \frac{N(q-1)}{2-q}$. In particular, the a priori estimate reads as below:
\begin{theorem}\label{sapprinc}
Assume \eqref{A1}, \eqref{A2}, \eqref{H} with $q=p-\frac{N}{N+1}$ and $u_0\in L^{1+\om}$, $0<\omega<1$. Moreover, let $u$ be a renormalized solution \eqref{P} as in \cref{defrin} and satisfying the regularity condition
\[
(1+|u|)^{-\frac{1-\omega}{2}}u\in L^2(0,T;H_0^1(\Omega)).
\]
Then we have that the following estimate holds:
\begin{equation}\label{disapprinc}
\sup_{t\in [0,T]}\|u(t)\|_{L^{1+\omega}(\Omega)}^{1+\omega}+\| D((1+|u|)^{-\frac{1-\omega}{2}-1}u)\|_{L^2(Q_T)}^2\le M 
\end{equation}
where the constant $M$ depends on $\alpha$, $q$, $\gamma$, $N$, $|\Omega|$, $T$ and on $u_0$ and remains bounded when $u_0$ varies in sets which are bounded in $L^{1+\omega}(\Omega)$.
\end{theorem}

\section{Regularizing effects and long time decay}\label{deca}

We now refine the contraction estimate of \cref{contraction} into a time pointwise statement. 

\begin{lemma}\label{propW11} Assume \eqref{A1}, \eqref{H} with $2-\frac{N}{N+1}< q<2$ and \eqref{ID1}. Let $u$ be a solution of \eqref{P}, and $\Phi: R\to R$ be a $C^2$,   convex  function such that $\Phi'(0)=0$ and $\Phi''(\xi)\leq c (1+|\xi|)^{\si-2}$. Then the function 
$t\mapsto \int_{\Omega}\Phi(u(t))\,dx$ belongs to $W^{1,1}(0,T)$ and satisfies 
\be\label{pointw}
 \frac{\text{d}}{\text{d}t}\int_{\Omega}\Phi(u(t))\,dx+ \int_{\Omega}a(t,x,u, D u)\cdot  D u \Phi''(u)\,dx=\int_\Omega H(t,x, D u)\Phi'(u)\,dx
\ee
for a.e. $t\in(0,T)$.
\end{lemma}

\proof 
We choose $\psi(u)= \Phi'(T_n(u))$ in \cref{propint}, using \rife{H} and the growth of $\Phi(\cdot)$ we  get
\besac
\ds\int_{\Omega}S_n(u(t_2))\,dx+
\int_{t_1}^{t_2} \into a(t,x,u, D u)\cdot D T_n(u) \,\Phi''(T_n(u))\,dx\,dt
\\
[3mm]\ds\le c  \int_{t_1}^{t_2} \into   | D u|^q  \, (1+|T_n(u)|)^{\si-1} \,dx\,dt  +\int_{\Omega}S_n(u(t_1))\,dx
,
\eesac
\noindent
where $S_n(v)=\int_0^v \Phi'(T_n(z))\,dz$.\\
\cref{sap} provides us with a bound for the right-hand side which is uniform in $n$ and independent of the interval $(t_1,t_2)$, whereas the left-hand side is positive. Hence, we let $n$ tend to infinity and apply Fatou's Lemma to deduce that $a(t,x,u, D u) \cdot  D u\, \Phi''(u)\in L^1(0,T;L^1(\Omega))$.\\
Now, if we change the previous test function into $\xi \,\Phi'(T_n(u))$ where $\xi\in C_c^{\infty}(0,T)$, we obtain
\besac
\ds-\int_0^T \xi'\int_{\Omega} S_n(u)\,dx\,dt+ \int_0^T \xi\int_{\Omega} a(t,x,u, D u) \cdot D T_n(u) \,\Phi''(T_n(u))\,dx\,dt
\\
[3mm]\ds=\int_0^T\xi\int_\Omega H(t,x, D u)\,\Phi'(T_n(u))\,dx\,dt
\eesac
for every $\xi\in C_c^{\infty}(0,T)$. Otherwise, this means that we have
\begin{equation}\label{w11}
\begin{array}{c}
\ds
\frac{\text{d}}{\text{d}s}\int_\Omega S_n(u)\,dx=
-  \int_{\Omega} a(t,x,u, D u) \cdot D T_n(u)\, \Phi''(T_n(u))\,dx\\
[3mm]\ds
+\int_\Omega H(t,x, D u)\,\Phi'(T_n(u))\,dx
\end{array}
\end{equation}
in the sense of distributions.\\
Since $a(t,x,u, D u)\, \cdot  D u \,\Phi''(u) \in L^1(0,T;L^1(\Omega))$, and the same holds for $|Du|^q\, \Phi'(u)$, the right-hand side of \rife{w11} converges in $L^1(0,T)$ as  $n\to \infty$ by Lebesgue's Theorem. We deduce that
$\frac{\text{d}}{\text{d}s}\int_\Omega S_n(u)\,dx$ converges to $\frac{\text{d}}{\text{d}s}\int_\Omega \Phi(u(s))\,dx$ in $L^1(0,T)$. Finally, we get that $\int_\Omega \Phi(u(s))\,dx\in W^{1,1}(0,T) $ and the equality \rife{pointw}  
holds for almost every $t\in (0,T)$.
\endproof

\begin{proposition}\label{corexp}
Assume \eqref{A1},   \eqref{H} with $2-\frac{N}{N+1}< q<2$ and \eqref{ID1}. Let $u$ be a solution of \eqref{P}. There exists  a value $k_0>0$ such that, for $k\geq k_0$,   $u$ satisfies
\begin{equation}\label{sec}
\frac{\text{d}}{\text{d}t}\int_{\Omega}|G_k(u(t))|^{\si}\,dx+\alpha \frac{2(\si-1)}{\si}
\int_{\Omega}  | D [|G_k(u)|^{\frac{\si}{2}}] |^2  \,dx\,ds
\le 0
\end{equation}
a.e. $t\in(0,T)$ and for all $k\ge k_0$.\\
In particular,  the following exponential decay holds:
\[
\|G_k(u(t))\|_{L^{\si}(\Omega)}^{\si}< e^{-\lambda t} \|G_k(u_0)\|_{L^{\si}(\Omega)}^{\si}\,\,\, \forall \, t\in(0,T)
\]
for some  $\lambda>0$  and $k\geq   k_0$.
\end{proposition}

\begin{remark}
Defining
\[
\eta(k)=\int_{ \{|u_0|>k \}}|u_0|^{\si}\,dx\,,
\]
 a possible choice of $k_0$ can be given by  
\[
k_0=\inf\left\{k:\,\, \eta(k)\le \frac{\delta_0}{2^{\frac N{2-q}}}  \right\}
\]
where $\delta_0$ is the  value given by \cref{lemGk}, see \cref{precise}. Note that $\eta(k)$ is a non increasing function of $k$. 
\end{remark}

\proof
If $\si\geq 2$, we apply \cref{propW11} with $\Phi'(\xi)= |G_k(\xi)|^{\si-2}G_k(\xi)$ and we get,
thanks to the assumptions \eqref{A1} and to \eqref{H}, 
$$
\frac1{\si}\frac{\text{d}}{\text{d}t}\int_{\Omega}|G_k(u(t))|^{\si}\,dx+
\alpha(\si-1)
\into | D  G_k(u) |^2 | G_k(u)|^{\si-2}\,dx \\
\le\gamma  \into  | D G_k(u)|^q | G_k(u)|^{\si-1}\,dx\,. 
$$
Next we proceed  as in \cref{lemGk}. We take $\bar \delta_0$ such that  $\frac{c_S 2^q \gamma}{\si^q} {\bar \delta_0}^{\frac{2-q}{N}}=2\alpha\frac{\si-1}{\si^2}$ and we let $  k_0$ be correspondingly defined so that $\into |G_k(u_0)|^{\si}\,dx \leq \bar\de_0$ for every $k\geq k_0$. This implies that $\int_{\Omega}|G_k(u(t))|^{\si}\,dx\leq \bar \de_0$ for all $t\in (0,T)$. Eventually, we obtain  a pointwise in time version of \rife{lim}, which leads to  
\begin{equation*}
\frac{\text{d}}{\text{d}s}\|G_k(u(s))\|_{L^{\si}(\Omega)}^{\si}+\frac{ 2\alpha(\si-1)}{\si}  \| D[|G_k(u(s))|^{\frac{\si}{2}}] \|_{L^2(\Omega)}^2\le 0
\end{equation*}
for a.e. $t\in [0,T]$ and for all $k\ge k_0$.
\\
By the Poincar\'e inequality, we deduce the following differential inequality:
\begin{equation*}
\frac{\text{d}}{\text{d}s}\|G_k(u(s))\|_{L^{\si}(\Omega)}^{\si}+2c_P\frac{\alpha(\si-1)}{\si}  \|G_k(u(s)) \|_{L^{\si}(\Omega)}^{\si}\le 0 
\end{equation*}
for a.e. $t\in [0,T]$ and for all $k\ge k_0$. This  implies 
\[
 \|G_k(u(t))\|_{L^{\si}(\Omega)}^{\si}\le  e^{-\lambda t}\|G_k(u_0)\|_{L^{\si}(\Omega)}^{\si}<e^{-\lambda t}\bar \delta_0 \qquad \forall t\in (0,T)\,,
\]
where $\lambda=2c_P\frac{\alpha(\si-1)}{\si}$.

If $\si<2$, we take $\Phi'(\xi)= ( \si-1)\int_0^{G_k(\xi)} (\vep+s)^{ \si-2}ds$ in \cref{propW11}, then we proceed again as in \cref{lemGk} and we obtain the same conclusion after letting $\vep\to 0$.
\endproof

\begin{remark} According to the above estimate, we have found that $ \|G_k(u(t))\|_{L^{\si}(\Omega)}$ decays exponentially with the rate $\lambda= 2c_P\frac{\alpha( \si-1)}{\si^2}$. This rate will be actually improved later once we show the decay in the $L^2$-norm.

We also stress that the above exponential rate is depending on the bounded domain $\Omega$ (through Poincar\'e's inequality). However, the above proof would also lead to  a long time decay of the $L^{\si}$-norm   in unbounded domains. It is enough to use Sobolev's inequality instead of Poincar\'e and a polynomial decay follows. 
\end{remark}

\subsection{Regularizing effect $L^{\si}-L^{r}$}

\begin{proposition}\label{Prop2r}
Assume \eqref{A1},  \eqref{H} with $2-\frac{N}{N+1}< q<2$, \eqref{ID1} and let $u$ be a solution of \eqref{P}. Then, $u$ belongs to $C^0((0, T];L^{r}(\Omega))$ for any $r>\si$. Moreover, there exists a value $k_0$, independent of $r$, such that the following decay holds:
\be\label{gkr}
\|G_k(u(t))\|_{L^{r}(\Omega)}^{r}\le c_r\frac{\|G_k(u_0)\|_{L^{\si}(\Omega)}^{r}}{t^{N\frac{(r-\si)}{2\si}}}\quad\forall\,\,t\in(0,T),\,\,\forall \,k\ge k_0
\ee
where $c_r=c(r,q,\alpha,\gamma,N)$. Furthermore
\be\label{decayu}
\|u(t)\|_{L^{r}(\Omega)}^{r}\le \frac{c}{t^{N\frac{(r-\si)}{2\si}}}\quad\forall\,\,t\in(0,t_0]
\ee
where $c=c(r,q, \alpha,\gamma,N, t_0,u_0, |\Omega|)$.
\end{proposition}

\begin{remark}\label{k0}
Let us stress that the constant  in the estimate \rife{gkr} does not depend on the set $\Omega$. The value $k_0$ is given by  the smallness condition on  $\|G_k(u_0)\|_{\elle\si}$, hence it is uniform whenever $u_0$ varies in a  compact set of $\elle\si$.

The constant $c$ in \rife{decayu} depends on $u_0$ through its modulus of equi-integrability in $\elle\si$, in particular the constant can be chosen uniform if $u_0$ varies in a  compact subset of $\elle\si$.
\end{remark}
\proof
\noindent
{\it Step 1} \\
\noindent
We first prove the $G_k(u(t))$ decays in the $L^{r}(\Omega)$ norm, which in turn implies that $u\in L^{\infty}(0, T;L^{r}(\Omega))$.

Assume by now that $\si\geq 2$. For $r> \si $, we take $\Phi'(u)=\int_{0}^{G_k(u)}|T_n(v)|^{r-2}\,dv$ in \cref{propW11}, which is licit   for every fixed $n$. We have that
\begin{equation}\label{e}
\begin{array}{c}
\ds
\frac{\text{d}}{\text{d}s}\int_{\Omega}\biggl[\int_0^{G_k(u(s))}\biggl(\int_0^v|T_n(z)|^{r-2}\,dz\biggr)\,dv\biggr]\,dx+
\alpha
\int_{\Omega} | D G_k(u(s))|^2 |T_n(G_k(u(s)))|^{r-2}\,dx\\
[3mm]\ds
\le \gamma   \int_{\Omega} | D G_k(u(s))|^q \biggl(\int_{0}^{G_k(u(s))}|T_n(v)|^{r-2}\,dv\biggr)\,dx \quad\text{a.e.}\,\,s\in (0,T].
\end{array}
\end{equation}
We first estimate the test function itself by H\"older's inequality with indices $\left(\frac{1}{2-q},\frac{1}{q-1}\right)$, getting
\begin{align*}
\int_{0}^{G_k(u(s))}|T_n(v)|^{r-2}\,dv & \le |T_n(G_k(u(s)))|^{q(\frac{r}{2}-1)}\int_{0}^{G_k(u(s))}|T_n(v)|^{(\frac{r}{2}-1)(2-q)}\,dv\notag\\
&\le |T_n(G_k(u(s)))|^{q(\frac{r}{2}-1)}\biggl(\int_{0}^{G_k(u(s))}|T_n(v)|^{\frac{r}{2}-1}\,dv\biggr)^{2-q}|G_k(u(s))|^{q-1}.
\end{align*}
We thus estimate the r.h.s. of \eqref{e} as below
\beac\label{3H}
\ds\int_{\Omega} | D G_k(u(s))|^q \biggl(\int_{0}^{G_k(u(s))}|T_n(v)|^{r-2}\,dv\biggr)\,dx\\
[3mm]\ds\le \int_{\Omega} | D G_k(u(s))|^q |T_n(G_k(u(s)))|^{q(\frac{r}{2}-1)}\biggl(\int_{0}^{G_k(u(s))}|T_n(v)|^{\frac{r}{2}-1}\,dv\biggr)^{2-q}|G_k(u(s))|^{q-1}\,dx\\
[3mm]\ds\le \biggl(\int_{\Omega} | D G_k(u(s))|^2 |T_n(G_k(u(s)))|^{r-2}\,dx\biggr)^{\frac{q}{2}}\biggl(\int_{\Omega}\biggl(\int_{0}^{G_k(u(s))}|T_n(v)|^{\frac{r}{2}-1}\,dv\biggr)^{2^*}\,dx\biggr)^{\frac{2-q}{2^*}}\times\\
[3mm]\ds\times\biggl(\int_{\Omega}|G_k(u(s))|^{\si}\,dx\biggr)^\frac{2-q}{N}\\
[3mm]\ds
\le {c_S} \biggl(\int_{\Omega} | D G_k(u(s))|^2 |T_n(G_k(u(s)))|^{r-2}\,dx\biggr)\|G_k(u(s))\|_{L^{\si}(\Omega)}^{q-1},
\eeac
where we have also applied the H\"older inequality with three exponents $\left(\frac{2}{q},\frac{2^*}{2-q},\frac{N}{2-q}\right)$ and then Sobolev's embedding.
\\
Let $\de_0$ be given by \cref{lemGk}. Then we fix a value $\bar \de<\de_0$ so that $\alpha-{\gamma}{c_S}{\bar\de}^{\frac{q-1}{\si}}>0$, e.g. we take $\bar \de^{\frac{q-1}{\si}}< \min\left( \frac{\alpha}{2\gamma c_S}, \de_0^{\frac{q-1}{\si}}\right)$,   and we take $k_0$ such that $\|G_{ k}u_0\|_{L^{\si}(\Omega)}^{\si}<\bar \delta$ for every $k\ge k_0$. Without loss of generality we may also assume, from \cref{corexp}, that $\|G_k(u(s))\|_{L^{\si}(\Omega)}$ is nonincreasing in time for $k\ge k_0$. 
Since by \cref{lemGk} we have $\|G_k(u(t))\|_{L^{\si}(\Omega)}<\bar \de$ for every $t$, \rife{e}--\rife{3H} imply 
$$
\begin{array}{c}
\ds
\frac{\text{d}}{\text{d}s}\int_{\Omega}\biggl[\int_0^{G_k(u(s))}\biggl(\int_0^v|T_n(z)|^{r-2}\,dz\biggr)\,dv\biggr]\,dx+
\alpha
\int_{\Omega} | D G_k(u(s))|^2 |T_n(G_k(u(s)))|^{r-2}\,dx\\
[3mm]\ds
\le \gamma\, c_S\, \bar \de^{\frac{q-1}\si}  \int_{\Omega} | D G_k(u(s))|^2 |T_n(G_k(u(s)))|^{r-2}\,dx 
\\
[3mm]\ds
\le \frac\alpha 2 \int_{\Omega} | D G_k(u(s))|^2 |T_n(G_k(u(s)))|^{r-2}\,dx 
\end{array}
$$
due to the choice of $\bar\de$. Then subtracting the right-hand side and using Sobolev's inequality we deduce 
\begin{equation}\label{d}
\frac{\text{d}}{\text{d}s}\int_{\Omega}\biggl[\int_0^{G_k(u(s))}\biggl(\int_0^v|T_n(z)|^{r-2}\,dz\biggr)\,dv\biggr]\,dx+
\frac{\al}{2c_S}\biggl[\int_{\Omega} \biggl|\biggl( \int_0^{G_k(u(s))} |T_n(v)|^{\frac{r}{2}-1}\,dv \biggr)\biggr|^{2^*}\,dx\biggr]^{\frac{2}{2^*}}
\le 0
\end{equation}
for every $k\ge k_0$. We set
\[
\varPhi_n(v)=\int_0^v |T_n(z)|^{\frac{r}{2}-1}\,dz
\qquad \text{and} \qquad
\varTheta_n(v)=\int_0^v \biggl(\int_0^w|T_n(z)|^{r-2}\,dz\biggr)\,dw
\]
so that \eqref{d} reads as
\[
\frac{\text{d}}{\text{d}s}\int_{\Omega} \varTheta_n(G_k(u(s)))\,dx+c\biggl(\int_{\Omega}|\varPhi_n(G_k(u(s)))|^{2^*}\,dx\biggr)^{\frac{2}{2^*}}\le 0 \qquad \text{a.e.}\,\, s\in (0,T],\,\,\forall\,\,k\ge k_0.
\]
We show now  the link between $\varTheta_n(v)$ and $\varPhi_n(v)$, in order to recover a suitable differential inequality.
To this aim, we rewrite  $r= \frac{2^*}{2}r\omega+ \si(1-\omega)$, where $\omega=\frac{2r-2\si}{2^*r-2\si}$. Note that $0<\omega<1$ since $\si<r$. Thus
\begin{align}
\int_0^{w} |T_n(z)|^{r-2}\,dz&=\int_0^{w}|T_n(z)|^{(\frac{2^*r}{2}-1)\omega}|T_n(z)|^{(\si-1)(1-\omega)-1}\,dz\notag\\
&\le \biggl(\int_0^{w}|T_n(z)|^{\frac{2^*r}{2}-1}\,dz\biggr)^{\omega} \biggl(\int_0^{w}|T_n(z)|^{\si-1-\frac{1}{1-\omega}}\,dz\biggr)^{1-\omega}\notag \\
&\le c_r\biggl(\int_0^{w} |T_n(z)|^{\frac r2-1}\,dz\biggr)^{\omega 2^*}|w|^{\si(1-\omega)-1}\label{arg2},
\end{align}
where we used the inequality
\[
\int_0^w|T_n(z)|^{\frac{2^*r}{2}-1}\,dz\le c_r \biggl(\int_0^w |T_n(z)|^{\frac{r}{2}-1}\,dz\biggr)^{2^*}
\]
in \eqref{arg2}.  Here and below, $c_r$ denotes possibly different constants which may depend on $r$ but are independent of $n$. Then, \eqref{arg2}  leads to
\begin{align*}
\Theta_n(G_k(u(s)))&=\int_0^{G_k(u(s))}\int_0^w |T_n(z)|^{r-2}\,dz\,dw\\&\le c_r\int_0^{G_k(u(s))} |\varPhi_n(w)|^{2^*\omega}\, |w|^{\si(1-\omega)-1} \,dw\\
&\le c_r |\varPhi_n(G_k(u(s)))|^{2^*\omega} |G_k(u(s))|^{\si(1-\omega)}
\end{align*}
which implies
\begin{align}
\int_{\Omega}\varTheta_n(G_k(u(s)))\,dx&\le c_r\int_{\Omega}\biggl(|\varPhi_n(G_k(u(s)))|^{2^*\omega}|G_k(u(s))|^{\si(1-\omega)}\biggr)\,dx\notag\\
&\le c_r \biggl(\int_{\Omega}|\varPhi_n(G_k(u(s)))|^{2^*}\,dx\biggr)^{\omega}\|G_k(u(s))\|_{L^{\si}(\Omega)}^{\si(1-\omega)}\label{arg3}\\
&\le c_r \biggl(\int_{\Omega}|\varPhi_n(G_k(u(s)))|^{2^*}\,dx\biggr)^{\omega}\|G_k(u_0)\|_{L^{\si}(\Omega)}^{\si(1-\omega)}\notag 
\end{align}
where we used   the nondecreasing character of  $\|G_k(u(s))\|_{L^{\si}(\Omega)}$  for $k\ge k_0$.\\
We summarize the previous steps as follows:
\begin{equation}\label{distheta}
\frac{\text{d}}{\text{d}s}\int_{\Omega} \varTheta_n(G_k(u(s)))\,dx+c_r\bigg[
\frac{\bigl(\int_{\Omega}\varTheta_n(G_k(u(s)))\,dx\bigr)}{\|G_k(u_0)\|_{L^{\si}(\Omega)}^{\si(1-\omega)}}\biggr]^{\frac{2}{2^*\omega}}\le 0
\end{equation}
for a.e. $s\in (0,T]$ and for all $k\ge k_0$.\\
Setting 
\[
y(s)=\int_{\Omega} \varTheta_n(G_k(u(s)))\,dx\quad\text{and}\quad
C_0=\|G_k(u_0)\|_{L^{\si}(\Omega)}^{\frac{2\si(1-\omega)}{2^*\omega}}
\]
we rewrite \eqref{distheta} in the equivalent  form
\[
y(s)'+\frac{c_r}{C_0}y(s)^{1+\rho}\le 0
\]
where the exponent $\ro$ is given by $\rho=\frac{2\si}{N(r-\si)}$. Let us recall that $y\in W^{1,1}(0,T)$ by \cref{propW11}. Thus, we integrate in the time variable  and we get (see e.g. \cite[Lemma 2.6]{P2})
\[
\int_{\Omega} \varTheta_n(G_k(u(t)))\,dx\le \frac{c_r}{t^{\frac{1}{\rho}}}\|G_k(u_0)\|_{L^{\si}(\Omega)}^{\frac{2\si(1-\omega)}{2^*\rho\omega}}
\]
for every $n$ and for every $0< t\le T$. We conclude letting $n$ go to infinity and using the value of $\omega$:
\[
\|G_k(u(t))\|_{L^{r}(\Omega)}^{r}\le c_r\frac{\|G_k(u_0)\|_{L^{\si}(\Omega)}^{r}}{t^{N\frac{(r-\si)}{2\si}}}\quad\forall\,\,t\in(0,T),\quad\forall \,\, k\ge k_0.
\]
In particular, we have proved that
\[
\|G_k(u(t))\|_{L^{r}(\Omega)}^{r}\le  \frac{c_r \de_0^{\frac{r}{\si}}}{t^{N\frac{(r-\si)}{2\si}}}\quad\forall\,\,t\in(0,T),\quad\forall \,\, k\ge k_0.
\]
Let us spend a  few words for the case that  $\si<2 $.   First of all, in this case it is enough to restrict the analysis  to $\si<r\leq 2$ since for larger exponents $r$ one can proceed as before.
Then, for $\si<r\leq 2$, we take $\Phi(u)=\int_{0}^{G_k(u)}(\vep+ |v|)^{r-2}\,dv$, we follow the previous steps and we obtain the same conclusion by letting $\vep\to 0$ (a similar argument is also detailed 
in \cref{renreg} in the \cref{app}).
\\

\noindent
{\it Step 2} \\
\noindent
The decay result written above and the decomposition $u=G_k(u)+T_k(u)$ provide us with the inequality 
\[
\|u(t)\|_{L^{r}(\Omega)}^{r}\le \frac{c_r\de_0^{\frac{r}{\si}}}{t^{N\frac{(r-\si)}{2\si}}}+k_0^{r}|\Omega|\quad\forall\,\,t\in(0,T),\quad\forall \,\, k\ge k_0
\]
which, for $t\leq t_0$, gives the following decay
\[
\|u(t)\|_{L^{r}(\Omega)}^{r}\le c\frac{\de_0^{\frac{r}{\si}}}{t^{N\frac{(r-\si)}{2\si}}}\quad\forall\,\,t\in(0,t_0)\quad\forall \,\, k\ge k_0
\]
where $c=c(r,k_0, |\Omega|,t_0)$.\\

\noindent
{\it Step 3} \\
\noindent
We conclude by showing that $u\in C^0((0,T];L^{r}(\Omega))$. Indeed, we already know that the  convergence $u(t_n)\to u(t)$ for $t_n\underset{n\to \infty}{\longrightarrow} t$ holds in $\elle{\si}$. Thanks to \rife{gkr}, it is easy to see that $u(t_n)$ is equi-integrable in $\elle{r}$;  we thus deduce the desired continuity regularity by Vitali's Theorem.
\endproof
\vskip0.3em

\begin{remark}
Note that the inequality \eqref{d} and the regularity $u\in L^{\infty}(0,T;L^{r}(\Omega))$ ensure $u\in L^{r}(0,T;L^{\frac{2^*r}{2}}(\Omega))$. Hence, we reason as in \cref{propW11} and deduce that $\int_\Omega |u(t)|^{r}\,dx\in W^{1,1}(0,T)$.
\end{remark}

\vskip1em

The next step  consists in showing  a similar regularizing effect in the $L^\infty$ norm. To this purpose, we simply observe  a   similarity  between \eqref{P}  and the power problem with linear growth:
\begin{equation}\label{Ppot}\tag{$P_{\text{pow}}$}
\begin{cases}
\begin{array}{ll}
\ds u_t- \text{div}(a(t,x,u, D u)) =\varPsi(t,x)u  &\ds \text{in}\ Q_T,\\
\ds u=0   &\ds \text{on}\ (0,T)\times \partial\Omega,\\
\ds u(0,x)=u_0(x)  &\ds  \text{in}\  \Omega,
\end{array}
\end{cases}
\end{equation}
where $\varPsi \in L^{\infty}(0,T;L^{\beta}(\Omega))$ for $\beta\ge \frac{N}{2}$. 

In fact, under conditions \rife{A1}, \rife{H}, equation \eqref{P} implies 
\besac
\ds\frac1{r}\frac{d}{dt} \int_{\Omega} |  u(t)|^{r}\,dx +
\alpha (r-1)
\int_{\Omega} | D  u(t)|^2 |  u(t)|^{r-2}\,dx
\le \gamma   \int_{\Omega} | D  u(t)|^q |  u(t)|^{r-1}\,dx 
\eesac
a.e. $t\in (0,T]$. Furthermore, since by Young's inequality and using $r\geq 2$ we have 
\besac
\ds\gamma   \int_{\Omega} | D  u(t)|^q |  u(t)|^{r-1}\,dx  \\
[3mm]\ds\leq \frac\alpha 2
\int_{\Omega} | D  u(t)|^2 |  u(t)|^{r-2}\,dx + c_{\gamma,\alpha} \into |u|^{r}\, |u|^{(q-1)\frac 2{2-q}}\, dx 
\\
[3mm]\ds
\le\frac{\alpha(r-1) } 2 \int_{\Omega} | D  u(t)|^2 |  u(t)|^{r-2}\,dx + c_{\gamma,\alpha} \into |u|^{r}\, |u|^{(q-1)\frac 2{2-q}}\, dx \,,
\eesac
we deduce that
\begin{equation}\label{2r}
\frac1{r}\frac{d}{dt} \int_{\Omega} |  u(t)|^{r}\,dx +
\frac{\alpha (r-1)}2
\int_{\Omega} | D  u(t)|^2 |  u(t)|^{r-2}\,dx\leq c_{\gamma,\alpha} \into |u|^{r}|u|^{\frac{2\si}{N}}\,dx\,,
\end{equation}
where we used $\frac{2(q-1)}{2-q}= \frac{2\si}N$.

The same kind of inequality clearly holds for \eqref{Ppot}, where we have directly
$$
\frac1{r}\frac{d}{dt} \int_{\Omega} |  u(t)|^{r}\,dx +
\alpha (r-1)
\int_{\Omega} | D  u(t)|^2  |u(t)|^{r-2}\,dx\leq   \into \, |\varPsi|\, |u|^{r}\,dx\,.
$$
In other words, problem \rife{P} shares some energy estimates with the linear problem \rife{Ppot} whenever $|u|^{\frac{2\si}{N}}$ plays the role of the function $\varPsi(\cdot)$.  We now  observe  that $|u|^{\frac{2\si}{N}}\in \limitate \beta $ for some $\beta>\frac N2$ as soon as $u\in \limitate r$ with $r>\si$.

\vskip0.4em
Problem \rife{Ppot} was  studied in detail in \cite{P2} as a preliminary tool for parabolic problems with superlinear powers as \rife{power}. By means of the above remark, we can use directly those results to conclude with our regularizing effects. 
More precisely, we can apply the following Lemma, which is nothing but  a consequence of  a  classical Moser iteration method. 
We won't give the proof of the result below since it is contained\footnote{We warn the reader that this result is not stated as we do here  in \cite{P2}. However, the   inequality \rife{preMoser} corresponds to (2.27) in the given reference, which is the starting point of the iteration argument leading to the $L^\infty$ bound. It makes no difference whether \rife{preMoser} is required as an assumption - as we state here - or whether it is proved to be a preliminary property of solutions - as in \cite[Proposition $2.7$]{P2}} in \cite[Proposition $2.7$]{P2}

\begin{lemma}\label{moser} Let $v\in \paraccaloc\cap \continue m$  satisfy $v\in L^{r}_{\text{\rm loc}}(0,T;$ $  \elle r)$ for every
$r>1$ and, for every $s>m$, 
\be\label{preMoser}
\frac{d}{dt} \into |v|^{s}\,dx  +\Big( \into
|v|^{s\frac{N}{N-2}}\,dx\Big) ^{\frac {N-2}{N}}
\leq \kappa \,s\,\into |\psi|\,|v|^{s}\,dx\,,\qquad t\in (0,T)\,,
\ee
for some $\psi\in \limitate\beta$ with $\beta>\frac N2$, and some $\kappa$ independent of $s$ and $v$. Then we have $v(t)\in \elle\infty$ and 
$$
\|v(t)\|_{\elle\infty}
\leq Ce^{Ct\|\psi\|_{\limitate\beta}^{\frac{2\beta}{2\beta-N}}}\,t^{-
\frac N{2m}}\,\|v_0\|_{\elle m}\,,
$$
for every $t$ in $(0,T]$, where $C= C(m,\beta, N, \kappa)$. 
\end{lemma}

In virtue of the above Lemma, the Corollary below immediately follows from \cref{Prop2r}.

\begin{corollary}\label{bd}
Assume \eqref{A1},  \eqref{H} with $2-\frac{N}{N+1}< q<2$, \eqref{ID1} and let $u$ be a solution of \eqref{P}. Then $u$ is bounded in $L^{\infty}((\tau,T)\times\Omega)$, $\tau>0$. Moreover, if $k_0$ is the value given by \cref{Prop2r}, we have
\be\label{gkdecay}
\|G_k(u(t))\|_{L^{\infty}(\Omega)} \le C\frac{\|G_k(u_0)\|_{L^{\sigma}(\Omega)}}{t^{\frac N{2\si}}}\quad\forall\,\,t\in(0,T),\,\,\forall \,k\ge k_0
\ee
where $C=C(q,\alpha, \gamma, N)$.

Consequently, we have
\be\label{decayuinf}
\|u(t)\|_{L^{\infty}(\Omega)} \le \frac c{t^{\frac N{2\si}}}\quad\forall\,\,t\in(0,t_0),
\ee
where $c=c(q, \alpha,\gamma,N, t_0,u_0)$ is a constant which remains bounded when $u_0$ varies in a  compact set of $\elle\sigma$.
\end{corollary}

\proof
From \cref{Prop2r} we know that $u\in C((0,T];L^r(\Omega)$ for every $r<\infty$, which implies that the conclusion of  \cref{propW11} holds for functions $\Phi(u)$ having any possible polynomial growth. In particular, using $\Phi(u)= |G_k(u)|^{r-1}$ we obtain
\besac
\ds\frac1{r}\frac{d}{dt} \int_{\Omega} | G_k(u(t))|^{r}\,dx +
\alpha (r-1)
\int_{\Omega} | D G_k(u(t))|^2 | G_k(u(t))|^{r-2}\,dx
\\
[3mm]\ds
\le \gamma   \int_{\Omega} | D G_k(u(t))|^q | G_k(u(t))|^{r-1}\,dx \quad\text{a.e.}\,\,t\in (0,T],
\eesac
which yields, in the same way as we derived for \rife{2r}, 
\besac
\ds\frac{d}{dt} \frac1{r}\int_{\Omega} | G_k(u(t))|^{r}\,dx   +
\frac{\alpha (r-1)}{2}
\int_{\Omega} | D G_k(u(t))|^2 | G_k(u(t))|^{r-2}\,dx
\\
[3mm]\ds
\leq c_{\gamma,\alpha} \into |G_k(u)|^{r}|G_k(u)|^{\frac{2\si}{N}}\,dx\,.
\eesac
Using Sobolev's inequality in the left-hand side we obtain
$$
\frac{d}{dt}\frac1{r}\int_{\Omega} | G_k(u(t))|^{r}\,dx  +
\frac{\alpha 2(r-1)}{r^2c_S} \left( \into |G_k(u)|^{ \frac{2^*r}{2}}\,dx\right)^{\frac2{2^*}}
\leq C_{\gamma,\alpha} \into |G_k(u)|^{r}|G_k(u)|^{\frac{2\si}{N}}\,dx\,.
$$
Therefore, the function $v= G_k(u)$ satisfies \rife{preMoser} and we can apply \cref{moser} with $v= G_k(u)$, $m=\si$, $\psi=  |G_k(u)|^{\frac{2\si}{N}}$ and $\beta=\frac {Nr_0}{2\si}$ for some $r_0>\si$. We apply this estimate in the time interval $(\frac t2,t)$ for $k\geq k_0$; using \cref{Prop2r}, we have
$$
\|\psi(s)\|_{\elle\beta)}^{\frac{2\beta}{2\beta-N}}= \|G_k(u(s))\|_{\elle{r_0}}^{\frac{2\sigma r_0}{N(r_0-\si)}} \leq C \, \frac{\|G_k(u_0)\|_{\elle{\sigma}}^{\frac{2\si r_0}{N(r_0-\si)}}}{s} 
$$
hence, if $s\in (\frac t2,t)$
$$
\|\psi\|_{L^\infty(t/2,t);\elle\beta)}^{\frac{2\beta}{2\beta-N}}\leq \frac C t
$$
where $C$ may be chosen only depending on $q,N, \alpha, \gamma$ (e.g. we use $r_0= \si+1$ and $k_0$ satisfies $\|G_k(u_0)\|_{\elle{\sigma}}\leq \bar \de_0$ for some $\bar \de_0$ only depending on $q,N, \alpha, \gamma$).

Using \cref{corexp}  we conclude that
\begin{align*}
\|G_k(u(t))\|_{\elle\infty} &\leq C\, e^{Ct\|\psi\|_{L^\infty((t/2,t);\elle\beta)}^{\frac{2\beta}{2\beta-N}}} t^{-\frac N{2\sigma}}\|G_k(u(t/2))\|_{\elle\si} 
\\ & \leq C\,t^{-\frac N{2\sigma}}\|G_k(u(t/2))\|_{\elle\si}
\\
& \leq C\,t^{-\frac N{2\sigma}}\|G_k(u_0)\|_{\elle\si}\,.
\end{align*}
Finally, \rife{decayuinf} is obtained by the obvious inequality
\begin{align*}
\|u(t)\|_{\elle\infty} & \leq  \|G_{k_0}(u(t))\|_{\elle\infty} + k_0 
\\
& \leq \|G_{k_0}(u(t))\|_{\elle\infty} + k_0\,\left( \frac {t_0}t\right)^{\frac N{2\sigma}}\qquad \forall \, t\leq t_0\,.
\end{align*}
\endproof

We observe that estimate \rife{gkdecay} contains the same decay  of the $L^\infty$-norm as for the heat equation or, similarly, for nonlinear coercive operators. Indeed, this estimate could have been as well obtained through the nonlinear methods used in \cite{Po}, rather than using \cref{moser}.  

\subsection{Asymptotic behaviour }

This Subsection is devoted to the analysis of the behaviour of the solution for large times.\\

\begin{proposition}\label{propdec}
Assume \eqref{A1},   \eqref{H} with $2-\frac{N}{N+1}< q<2$ and \eqref{ID1}. Let $u$ be a solution of \eqref{P}.  Then, we have
\[
\lim_{t\to\infty}\|u(t)\|_{L^{\infty}(\Omega)}=0.
\]
\end{proposition}
\proof
We start recalling the following decay for $G_k(u)$ in the $L^{\infty}$-norm, which is given by \rife{gkdecay}:
\begin{equation}\label{gk}
\|G_k(u(t))\|_{L^{\infty}(\Omega)}\le c (t-\tau)^{-\frac{N}{2\si}} \|G_k(u(\tau))\|_{L^{\sigma}(\Omega)},
\end{equation}
where $t>\tau>0$ and $k\geq k_0$. We recall that $k_0$ is given by requiring  $\|G_k(u(\tau))\|_{L^{\sigma}(\Omega)}<\delta $ for $k\ge k_0$, where $\de$ is a value only depending on $q, N, \alpha, \gamma$. In particular, we can take any $k>0$ such that $k> \|u(\tau)\|_{L^{\infty}(\Omega)}-\delta $. 
\\
Now, suppose by contradiction that 
\be\label{ell}
\lim_{t\to\infty}\|u(t)\|_{L^{\infty}(\Omega)}=\ell,
\ee
for some strictly positive $\ell$. The above limit exists thanks to \cref{corinfty}. 
We rewrite $u$ as the sum $u=G_k(u)+T_k(u)$ and estimate as follows:
\begin{equation*}
\begin{split}
\|u(t)\|_{L^{\infty}(\Omega)}&\le \|G_k(u(t))\|_{L^{\infty}(\Omega)}+k\\
&\le  c (t-\tau)^{-\frac{N}{2\si}} \delta+k
\end{split}
\end{equation*}
for $k\ge k_0$. Now, we fix a positive $\eps$ smaller than $\min\{\ell,\delta \}$: in this way, the decay \eqref{gk} holds for every  $k>\|u(\tau)\|_{L^{\infty}(\Omega)}-\eps$. In particular, we are allowed to choose  $k>0$ so that
\[
\|u(\tau)\|_{L^{\infty}(\Omega)}-\eps<k<\|u(\tau)\|_{L^{\infty}(\Omega)}-\frac{\eps}{2}\,,
\]
and such a choice is possible for arbitrarily large $\tau$ thanks to \rife{ell} and  since $\ell>0$.

Now, we set $\tau=\frac{t}{2}$, hence the previous estimate gives
\[
\|u(t)\|_{L^{\infty}(\Omega)}\le  c t^{-\frac{N}{2\si}} \delta +\|u(t/2)\|_{L^{\infty}(\Omega)}-\frac{\eps}{2}
\]
and, letting $t$ tend to infinity, we find a  contradiction.
\endproof

\begin{remark}\label{uniform}
It can be interesting to notice that the limit stated by \cref{propdec} is actually uniform whenever $u_0$ varies in a  compact set of $\elle\sigma$. Indeed, assume that $\{u_{0,n}\}_n\subset \elle\sigma$ is a sequence which is relatively compact in $\elle\sigma$. In particular, it is equi-integrable, and as a consequence of \cref{bd} (and \cref{k0} for the value of $k_0$ in \rife{gkdecay}) the sequence of solutions $\{u_n(t)\}_n$ is uniformly bounded in $\elle\infty$ for any $t>0$. Therefore, one has that $\sup\limits_n \|u_n(t)\|_\infty$ is itself  a non increasing quantity and admits a limit at infinity. A similar reasoning as in the  proof of \cref{propdec}  shows that  this limit is zero. 
\end{remark}

This information allows us to improve and extend the decay results previously obtained by using the smallness of $\|u(t)\|_{L^{\infty}(\Omega)}$ (for $t$ large) rather than the smallness of $\|G_k(u(t))\|_{L^{\sigma}(\Omega)}$ (for $k$ great enough). The result below provides a  sharp decay for the $L^\infty$-norm, extending a  similar result obtained for the heat operator in  \cite{BDL}, \cite{PZ}.

\begin{corollary}\label{sharp}
Assume \eqref{A1},   \eqref{H} with $2-\frac{N}{N+1}< q<2$ and \eqref{ID1}. Let $u$ be a solution of \eqref{P}. Then we have the following exponential decay for the $L^{\infty}(\Omega)$ norm of $u$, namely for every $\tau>0$
\be\label{infdecay}
\|u(t)\|_{L^{\infty}(\Omega)}\le c_\tau e^{-\lambda_1 t} \qquad \forall\, t>\tau>0
\ee
where $\lambda_1$ is defined as below:
\[
\lambda_1=\inf\biggl\{ \frac12 \frac{\integrale a(x,u, D u)\cdot  D u\,dx}{\|u\|_{L^2(\Omega)}^2}\,,\quad u\in H^1_0(\Omega) \biggr\}
\]
and the constant $c_\tau= c(\tau, \alpha, \gamma, q, N, \Omega, u_0)$ remains bounded whenever  $u_0$ varies in a  compact set of $\elle\sigma$.
\end{corollary}
\proof 
By \cref{propdec}, there exists $t_0$ such that 
$$
\|u(t)\|_{\elle\si} \leq \|u(t)\|_{\elle\infty} |\Omega|^{\frac1\si} < \de_0\qquad \forall t> t_0\,,
$$
where $\de_0$ can be chosen sufficiently small as needed in the proofs of  both \cref{corexp,Prop2r}. This means that  we can take $k_0=0$ in those Propositions  and we deduce that
\be\label{lambda}
\|u(t)\|_{\elle\si} \leq C_0 e^{-\lambda t}\qquad \forall t> t_0\,,
\ee
for some $\lambda>0$ and some $C_0>0$.

Thus, recalling \eqref{H}, we estimate as below
\besac
\ds\frac{1}{2}\frac{\text{d}}{\text{d}s}\integrale|u(s)|^2\,dx  + \int_\Omega  a(x,s,u(s), D u(s))\cdot D u(s)\,dx \\
[3mm]\ds
\le \gamma\int_\Omega| D u(s)|^q|u(s)|\,dx
\\
[3mm]\ds\le \gamma  \left(\int_\Omega | D u(s)|^2\,dx\right)^{\frac q2}\, \|u(s)\|_{L^{2^*}(\Omega)}^{2-q}\, \|u(s)\|_{L^{\sigma}(\Omega)}^{q-1}
\eesac
using H\"older's inequality with indices $\left(\frac{2}{q},\frac{2^*}{2-q},\frac{N}{2-q}\right)$ and the exact value of $\sigma= \frac{N(q-1)}{2-q}$. Using Sobolev's inequality and assumption \rife{A1} we deduce
$$
\frac{1}{2}\frac{\text{d}}{\text{d}s}\integrale|u(s)|^2\,dx   + \left(1- \frac{\gamma c_S}\alpha\, \|u(s)\|_{L^{\sigma}(\Omega)}^{q-1}\right) \int_\Omega  a(x,s,u(s), D u(s)) \cdot D u(s)\,dx \le 0\,.
$$
We now fix $\tau$ large enough so that $1- \frac{\gamma c_S}\alpha\, \|u(s)\|_{L^{\sigma}(\Omega)}^{q-1}>0$ for all $s\geq \tau$ and using the definition of $\lambda_1$ we get
\begin{equation}\label{improve}
\frac{\text{d}}{\text{d}s}\|u(s)\|^{2}_{L^{2}(\Omega)}+ 2\lambda_1 \left(1- \frac{\gamma c_S}\alpha\, \|u(s)\|_{L^{\sigma}(\Omega)}^{q-1}\right)\|u(s)\|^{2}_{L^{2}(\Omega)}\le 0.
\end{equation}
We use \eqref{lambda} in \eqref{improve}, thus we get
\begin{equation*}
\frac{\text{d}}{\text{d}s}\|u(s)\|^{2}_{L^{2}(\Omega)}+ 2\lambda_1(1-ce^{-\lambda (q-1)s})\|u(s)\|^{2}_{L^{2}(\Omega)}\le 0.
\end{equation*}
We set $g(s)=2\lambda_1(1-ce^{-\lambda (q-1)s})$. Then, by  integration we obtain
\[
\begin{split}
\|u(t)\|_{L^{2}(\Omega)}&\le  e^{-\frac12 \int_{\tau}^{t}g(s)\,ds}\|u(\tau)\|_{L^{2}(\Omega)}\\
&\le ce^{-\lambda_1 (t-\tau)}\|u(\tau)\|_{L^{2}(\Omega)}
\end{split}
\]
and thus we conclude the decay for the $L^2$-norm. We now use the regularizing effect and we deduce the same rate of exponential decay for all other Lebesgue norms and in particular for the $L^\infty$-norm. Finally, thanks to \cref{uniform}, the time $t_0$ chosen initially, as well as the other constants appearing later, can be uniformly chosen if $u_0$ varies in a compact set of $\elle\sigma$.
\endproof

\section{The case $1<q<2-\frac{N}{N+1}$ and $L^1$-data}\label{datil1}

In the range $1<q<2-\frac{N}{N+1}$, no restriction is needed on the Lebesgue class of initial data (notice that the value of $\sigma$ would become smaller than one for such $q$). 
We also stress that  one could as well consider  measure data, however we keep ourselves in the Lebesgue framework assuming $L^1$-data.

Since \cite{BlMu,Li}, it is customary for parabolic equations with $L^1$-data to work in the framework of renormalized solutions, as defined below. We recall that the space $\mathcal{T}_0^{1,2}(Q_T)$ is defined in \cite{BBGGPV}, see also \cref{defrin}.

\begin{definition}\label{defrin2}
We say that a function $u\in \mathcal{T}_0^{1,2}(Q_T)$ is a renormalized solution of \eqref{P} if satisfies
\begin{subequations} \label{RS2}
\renewcommand{\theequation}{$\widetilde{RS}$\arabic{equation}}
\begin{equation}
 H(t,x, D u)\in L^1(Q_T),\label{sr12} 
\end{equation}
\begin{equation}\label{sr22}
\begin{array}{c}
\ds
-\integrale S(u_0)\vp (0,x)\,dx+\iint_{Q_T}-S(u)\vp_t+a(t,x,u,D u)\cdot D (S'(u)\vp)\,dx\,dt\\
[3mm]\ds
=\iint_{Q_T}H(t,x,D u)S'(u)\vp\,dx\,dt,
\end{array}
\end{equation}
\begin{equation}\label{sr32}
\lim_{n\to \infty}\frac{1}{n}\iint_{\{n\le |u|\le 2n\}}a(t,x,u,D u)\cdot D u\,dx\,dt=0, 
\end{equation}
\end{subequations}
for every $S\in W^{2,\infty}(\mathbb{R})$ such that $S'(\cdot)$ has compact support, $\vp \in C^\infty_c([0,T)\times \overline \Omega)$ such that   $S'(u)\vp \in L^2(0,T;H_0^1(\Omega))$ (i.e., the product $S'(u)\vp$ is equal to zero on the parabolic boundary $(0,T)\times \partial \Omega$).
\end{definition}

\begin{remark}\label{cont1}
The spirit of \cref{rmkct} now applies to the truncations of $u$. Indeed, from the renormalized formulation we deduce that any $S$ such that $S'\in W^{1,\infty}(\R)$ and has compact support satisfies  $S(u)\in L^2(0,T; H^1_0(\Omega))$ so, if \rife{A2} holds, 
\begin{align*}
S(u)_t & =\text{div}(a(t,x,u, D u)S'(u))+ a(t,x,u,Du)Du S''(u)+H(t,x, D u)S'(u) 
\\
& \in L^2(0,T;H^{-1}(\Omega))+L^1(Q_T).
\end{align*}
This implies (\cite[Theorem $1.1$]{P1}) that   $S(u)\in C([0,T];L^1(\Omega))$ and in particular, from the formulation \rife{sr32},   $S(u)(0)= S(u_0)$ in $\elle1$. 

Later, we are going to see in \cref{equiv1} below that actually $u$ itself is continuous in $\elle1$.
\end{remark}

We point out that, by density arguments, one can afford for test functions $\vfi\in  L^2(0,T;H_0^1(\Omega))\cap  L^{\infty}(Q_T)$ such that $\vp_t\in L^{2}(0,T;H^{-1}(\Omega))$, $\vp(T,x)=0$. Moreover, integration by parts formulas also hold for the renormalized formulation, for instance  the  conclusion of \cref{propint}   holds   for every $t_1, t_2 \in [0,T]$ at least when $\psi'(\cdot)$ has compact support.  


In a  similar way as for the other cases, we stress that an equivalent formulation can be given by only requiring the global time continuity (in $L^1(\Omega)$) and a  standard energy formulation for positive times. Namely, we have

\begin{proposition}\label{equiv1} A function $u$ is a renormalized solution of \rife{P} (in the sense of \cref{defrin2}) if and only if $u\in \continue {1} \cap \paraccaloc$, $u(0)= u_0$ and $u$ satisfies
\begin{equation}\label{1loc}
 \iint_{Q_T} - u\varphi_t+ a(t,x,u,D u)\cdot  D \varphi\,dx\,dt =  \iint_{Q_T} H(t,x, D u)\varphi\,dx\,dt
\end{equation}
for every $\varphi\in C^\infty_c(Q_T)$.
\end{proposition}

We postpone the proof of \cref{equiv1} to the \cref{app}, and proceed by showing a  priori estimates and contraction principles for the solutions.

\begin{lemma}\label{lemGkL1}
Assume \eqref{A1}, \eqref{A2} and \eqref{H} with $1<q<2-\frac{N}{N+1} $. Let $u_0\in \elle1$ and let $u$ be a renormalized solution of \eqref{P}. For $\omega$ positive and sufficiently small so that $\frac{N(q-1+\omega)}{2-q}<1$, we set   $\Theta(s):=  \int_0^{|s|}  [1-\frac1{(1+ |r|)^{\om}}] dr$.

Then,  there exists a positive value $\delta_0$,   depending on $\omega$, $\Omega$, $\alpha$, $c_S$ and $\gamma$, such that, for every $k>0$ and for every $\delta<\delta_0$ satisfying
\[
\into \Theta(G_k(u_0))dx <\delta
\]
we have
\begin{equation}\label{lemgk2}
\into \Theta(G_k(u(t)))dx<\delta\quad \forall\, t\in[0,T].
\end{equation}
\end{lemma}

\proof  Observe that $\Theta(s)$ is a positive,  even and convex function such that $\Theta' \in W^{1,\infty}(\R)$ and $\Theta'(0)=0$; we use \cref{equiv1} and we take $ \Theta'(G_k(u)) $ as test function, this is justified as in   \cref{propint}. Using \rife{A1} and \rife{H}   we get
\besac 
\ds\into \Theta(G_k(u(t_2)))dx  + \alpha \int_{t_1}^{t_2}\into |DG_k(u)|^2\,\Theta''(G_k(u)) \,dx\,dt\\
[3mm]\ds\leq \into \Theta(G_k(u(t_1)))dx
+ \gamma \int_{t_1}^{t_2}\into |DG_k(u)|^q\,|\Theta'(G_k(u))|\,dx\,dt
\eesac
for every $t_1,t_2\in (0,T)$. We notice that $\Theta$ satisfies, for some constant $c$ (eventually depending on $\omega$):
\be\label{condthe}
 |\Theta'(s)| \leq c\, \Theta''(s)^{\frac q2} \left( \int_0^s \Theta''(r)^{\frac{2-q}2}dr\right) \,  (1+ s^{\omega})
 \qquad \forall s\in \R\,.
\ee
Thanks to \rife{condthe} we get
\besac 
\ds\into \Theta(G_k(u(t_2)))dx  + \alpha \int_{t_1}^{t_2}\into |DG_k(u)|^2\,\Theta''(G_k(u)) dx\,dt
\leq \into \Theta(G_k(u(t_1)))dx
\\
[3mm]\ds 
+ \gamma\, c  \int_{t_1}^{t_2}\into |DG_k(u)|^q\,\Theta''(G_k(u))^{\frac q2} \left( \int_0^{G_k(u)} \Theta''(r)^{\frac{2-q}2}dr\right)\,   (1+ G_k(u)^{\omega})\,dx\,dt
\eesac
and since $ \int_0^{G_k(u)} \Theta''(r)^{\frac{2-q}2}dr\leq \left( \int_0^{G_k(u)} \Theta''(r)^{\frac12}dr\right)^{2-q} \,\, G_k(u)^{q-1}$ we deduce
\besac 
\ds\into \Theta(G_k(u(t_2)))dx  + \alpha \int_{t_1}^{t_2}\into |DG_k(u)|^2\,\Theta''(G_k(u)) dx\,dt
\leq \into \Theta(G_k(u(t_1)))dx
\\
[3mm]\ds+ \gamma\, c  \int_{t_1}^{t_2}\into |DG_k(u)|^q\,\Theta''(G_k(u))^{\frac q2} \left( \int_0^{G_k(u)} \Theta''(r)^{\frac12}dr\right)^{2-q} \,\,  G_k(u)^{q-1} (1+ G_k(u)^{\omega})\,dx\,dt\,.
\eesac
Applying H\"older inequality  with the usual three exponents $\left(\frac 2q; \frac{2^*}{2-q}; \frac N{2-q}\right)$ we obtain
\besac 
\ds\into \Theta(G_k(u(t_2)))dx  + \alpha \int_{t_1}^{t_2}\into |DG_k(u)|^2\,\Theta''(G_k(u))\,dx\,dt   \leq \into \Theta(G_k(u(t_1)))dx
\\
[3mm]\ds+ \gamma\, c  \int_{t_1}^{t_2}\left(\into |DG_k(u)|^2\,\Theta''(G_k(u)) dx\,dt\right)^{\frac q2} \left( \into \left( \int_0^{G_k(u)} \Theta''(r)^{\frac12}dr\right)^{2^*} \right)^{\frac{2-q}{2^*}}\times 
\\
[3mm]\ds \qquad \qquad \times 
  \left(\into  [G_k(u)^{q-1}(1+ G_k(u)^{\omega})]^{\frac{N}{2-q}}\,dx\right)^{\frac{2-q}N}\, dt
\eesac
and then, by Sobolev inequality, 
\besac
\ds\into \Theta(G_k(u(t_2)))dx  + \alpha \int_{t_1}^{t_2}\into |DG_k(u)|^2\,\Theta''(G_k(u)) dx\,dt
\leq \into \Theta(G_k(u(t_1)))dx
\\
[3mm]\ds
+ \gamma\, c  \left(\sup\limits_{t\in [0,t_2]}  \left[\into  [G_k(u)^{q-1}(1+ G_k(u)^{\omega})]^{\frac{N}{2-q}}\,dx\right]^{\frac{2-q}N}\right)
\int_{t_1}^{t_2} \into |DG_k(u)|^2\,\Theta''(G_k(u)) dx\,dt\,.
\eesac
Since $q<2-\frac N{N+1}$,   we have $\frac{N(q-1)}{2-q}<1$, and the same holds for $\frac{N(q-1+\omega)}{2-q}$ because of the condition required on $\omega$. Thus
\begin{align*}
&  
\left(\into  [G_k(u)^{q-1}(1+ G_k(u)^{\omega})]^{\frac{N}{2-q}}\,dx\right)^{\frac{2-q}N}  \leq c \, 
\left\{ \left(\into  |G_k(u(t))|\,dx\right)^{q-1} + \left(\into  |G_k(u(t))|\,dx\right)^{q-1+\omega}\right\} 
\\
& \qquad \qquad \leq c \,   \max\left\{ \left[\into  \Theta(G_k(u(t)))\,dx\right]^{q-1+\omega}\,,\, \left[ \into  \Theta(G_k(u(t)))\,dx\right]^{\frac{q-1}2} \right\} 
\end{align*}
where we used that $\Theta(s) \geq c \min (s,s^2)$ for some constant $c>0$.  We conclude that
 \besac
\ds\into \Theta(G_k(u(t_2)))dx  + \alpha \int_{t_1}^{t_2}\into |DG_k(u)|^2\,\Theta''(G_k(u)) dx\,dt\leq \into \Theta(G_k(u(t_1)))dx
\\
[3mm]\ds 
+ \gamma\, c  \left(\sup\limits_{t\in [0,t_2]}  \max\left\{ \left[\into  \Theta(G_k(u(t)))\,dx\right]^{q-1+\omega}\,,\, \left[ \into  \Theta(G_k(u(t)))\,dx\right]^{\frac{q-1}2} \right\} \right)
\times
\\
[3mm]\ds
\qquad\qquad\qquad \times 
\int_{t_1}^{t_2} \into |DG_k(u)|^2\,\Theta''(G_k(u)) dx\,dt\,.
\eesac
We conclude with the usual continuity argument we used before; 
Let $\de_0$ be such that 
$$
\gamma \, c  \, \max\left\{\de_0^{q-1+\omega}\,,\, \de_0^{\frac{q-1}2} \right\}  \leq \alpha\,.
$$
Then, whenever   $\into  \Theta(G_k(u(t)))\,dx\leq \de <\de_0$ for all $t\leq \tau$, we deduce that $\into  \Theta(G_k(u(\tau)))\,dx\leq \de$ as well. Moreover, since $u\in \continue 1$, we also have $\Theta(G_k(u))\in \continue 1$. Therefore, a continuity argument leads to the conclusion.
%
%
%
%
\qed

As in the previous cases, we deduce the following consequences from the above Lemma.

\begin{theorem}\label{estL1}
Assume \eqref{A1},  \eqref{H} with $1<q< 2-\frac{N}{N+1} $ and let $u_0\in \elle1$. Let $u$ be a renormalized   solution of \eqref{P}. Then $u\in M^{\frac{N+2}{N}}(Q_T)$ and $|D u|\in M^{\frac{N+2}{N+1}}(Q_T)$ and the following estimate holds:
\be\label{aprio2}
\sup_{t\in [0,T]}\|u(t)\|_{L^{1}(\Omega)}+\||D u|\|_{M^{\frac{N+2}{N+1}}(Q_T)}+ \iint_{Q_T}\frac{|D u|^2}{(1+|u|)^{\de+1}}\,dx\,dt\le M
\ee
where $\de>0$, the constant $M$ depends on $u_0$, $T$, $|\Omega|$, $q$, $\gamma$, $N$, $\alpha$, $\delta$,  and remains bounded when $u_0$ varies in sets which are bounded and equi-integrable in $L^{1}(\Omega)$. 
\end{theorem}

\begin{proof}
\cref{lemGkL1}  yields an estimate for $u$ in $\limitate1$ (since the function $\Theta$ has linear growth)  and an estimate for $\iint_{Q_T}\frac{|D G_k(u)|^2}{(1+|G_k(u)|)^{1+\omega}}\,dx\,dt$ for $k$ sufficiently large (from the energy term).   Next one estimates $ |D T_k(u)|$ in $\elle2$  in a  similar way as in \cref{lemGk}.

Finally, once we have an estimate for $\iint_{Q_T}\frac{|D u|^2}{(1+|u|)^{1+\omega}}\,dx\,dt$ with $\omega$ arbitrarily small, and  the fact that $u\in \limitate1$,  it follows an estimate of $u$ in $L^r((0,T); W^{1,r}_0(\Omega))$ for every $r<\frac{N+2}{N+1}$ (see e.g. \cite{BG}). As a consequence, we have that $H(t,x, D u)\in \parelle1$. The final conclusion of the Marcinkiewicz estimate is then a consequence of known results for equations with $L^1$- data (see  e.g. \cite{ST}). 
\end{proof}

\begin{remark}\label{anche}
Let us stress that the proof of \cref{lemGkL1} is actually performed for solutions $u\in \continue1\cap \paraccaloc$; combining this with the proof of \cref{estL1} is a way to show that any $u\in \continue1\cap \paraccaloc$ which satisfies \rife{1loc} enjoys the global regularity stated in \cref{estL1}, in particular $u\in M^{\frac{N+2}{N}}(Q_T)$, $|D u|\in M^{\frac{N+2}{N+1}}(Q_T)$ and $T_k(u)\in L^2(0,T;H^1_0(\Omega))$ for all $k>0$.
\end{remark}

We also deduce the contraction in $\elle\infty$ for this range of $q$.

\begin{corollary}\label{corinftyr2}
Assume \eqref{A1},  \eqref{H} with $1< q<2-\frac{N}{N+1}$ and the initial datum $u_0$ belonging to $L^{\infty}(\Omega)$. Moreover, let $u$ be a solution of \eqref{P} in the sense of Definition \eqref{defrin2}. Then, we have that $u$ belongs to $L^{\infty}(Q_T)$. Moreover, the following estimate holds:
\[
\|u(t)\|_{L^{\infty}(\Omega)}\le \|u_0\|_{L^{\infty}(\Omega)}\quad \forall\, t\in[0,T].
\]
\end{corollary}
\begin{proof}
We use again \cref{lemGkL1} with  $k=\|u_0\|_{L^{\infty}(\Omega)}$. In that case we obtain \rife{lemgk2} for any $\de>0$ arbitrarily small, and letting $\de\to 0$  we conclude that $\into \Theta(G_k(u(t)))dx=0$. This implies $\|u(t)\|_{L^{\infty}(\Omega)}\le k=\|u_0\|_{L^{\infty}(\Omega)}$.
\end{proof}

Finally, we prove the regularizing effect, in particular  all solutions become $L^2(\Omega)$ at positive time and therefore the decay estimates can then be deduced by the previous sections.

\begin{proposition}\label{Proprin1}
Assume \eqref{A1}-\eqref{A2} and \eqref{H} for $1< q< 2-\frac{N}{N+1}$ and $u_0\in \elle1$. Moreover, let $u$ be a solution of \eqref{P} in the sense of \cref{defrin2}. Then $\in C^0((0,T);\elle r)$ for every $r>1$ and   there exists a value $k_0$, independent of $r$, such that
\be\label{gkr1}
\|G_k(u(t))\|_{L^{r}(\Omega)}^{r}\le c_r\frac{\|G_k(u_0)\|_{L^{1}(\Omega)}^{r}}{t^{N\frac{(r-1)}{2 }}}\quad\forall\,\,t\in(0,T),\,\,\forall \,k\ge k_0
\ee
where $c_r=c(r,q,\alpha,\gamma,N)$. Furthermore
\be\label{decayu1}
\|u(t)\|_{L^{r}(\Omega)}^{r}\le \frac{c}{t^{N\frac{(r-1)}{2 }}}\quad\forall\,\,t\in(0,t_0]
\ee
where $c=c(r,q, \alpha,\gamma,N, t_0,u_0, |\Omega|)$. Finally, we have that $u\in L^\infty((\tau,T)\times \Omega))$ for any $\tau >0$ and satisfies
\[
\|u(t)\|_{L^{\infty}(\Omega)}\le c_\tau e^{-\lambda_1 t} \qquad \forall\, t>\tau>0\,.
\]
\end{proposition}

\proof The proof of \cref{renreg} shows how to get \rife{gkr1} for $1<r\leq 2$.  The case $r>2$ can be proved in the same way and actually even more simply; indeed,  for $r>2$ one can  use as test function $S'(\eta)= \int_0^\eta   |T_n(\xi)|^{r-2}d\xi$ which is allowed once we know that $u\in \paraccaloc$. Eventually, one ends up with estimates \rife{gkr1} and \rife{decayu1}. Once $u$ has been regularized in any $L^p$ space,  it is   possible to use what proved in the previous sections to conclude that the $L^\infty$-norm decays as $t\to \infty$. 
\endproof

\section{The case $N=2$ and $N=1$.}\label{N=}

This short section is devoted to the case of low dimensions $N=1$ and $N=2$. Of course, the arguments used before  also apply to those cases, provided the thresholds are modified according to the  peculiarity of the Sobolev  embedding of $H^1_0(\Omega)$ for $N=1,2$. To be precise,   the triplet of exponents used  for the H\"older inequality which  estimates the right-hand side,     both in \rife{psiu}  \cref{lemGk})  and in \rife{3H} (\cref{Prop2r}) should more generally read as $\left(\frac 2q, \frac s{2-q}, \frac{2s}{(2-q)(s-2)}\right)$, where $s>2$ is an exponent of Sobolev embedding, i.e.  $\|v\|_{\elle s} \leq c_S \|v\|_{H^1_0(\Omega)}$. Next, the estimate needs the equality  $\sigma= \frac{2s}{s-2}\,\frac{q-1}
{2-q}$ in order to be closed.

The reader can check that, if $N>2$, then $s= 2^*$ is the optimal choice, the above triplet becomes  $\left(\frac 2q, \frac {2^*}{2-q}, \frac{N}{(2-q)}\right)$ as appears in our mentioned proofs, and this also yields the value of $\sigma$ appearing in \rife{ID1}.  

On the other hand,  if $N=2$, $H^1_0(\Omega)$ is embedded in all Lebesgue spaces (but not in $\elle\infty$), which means that $s$ can be chosen arbitrarily large. This means that the value of $\sigma$ can be arbitrarily close to $ \frac{2(q-1)}{2-q}$ but strictly bigger.

If rather $N=1$ one is allowed to take $s=\infty$ in Sobolev embedding,   the triplet used in the above estimate should now read as $\left(\frac 2q, \infty, \frac{2}{(2-q)}\right)$ and the optimal value of $\sigma$ becomes $\sigma= \frac{2(q-1)}{2-q}$.

Notice that, in  both cases $N=1$ and $N=2$, the threshold of $L^1$-data becomes $q= \frac 43$. 

Once the value of $\sigma$ is modified, consistently with Sobolev's embedding as explained above, the rest of the proofs apply without changes. In the end, our results take the  following form in the case $N=1,2$.

\begin{theorem} Assume that \rife{A1}-\rife{A2} and \rife{H} hold true.

If $u_0$ satisfies 
$$
u_0 \in  \elle{\sigma}\,,\qquad \hbox{with}\quad
 \begin{cases}   
 \sigma >\ds \frac{2(q-1)}{2-q} & \hbox{if $N=2$ and $q\geq \frac43$}
 \\
  \sigma =\ds \frac{2(q-1)}{2-q} & \hbox{if $N=1$ and $q>\frac43$}
 \end{cases}
$$
then there exists a  weak solution of \rife{P} (in the sense of \cref{def1}). In addition, any solution satisfies the same conclusions of \cref{Prop2r}, \cref{bd,sharp}. 

If rather  $q< \frac43$  and   $u_0\in \elle1$, then  there exists a renormalized  solution of \rife{P} (see \cref{defrin2}). In addition, any  solution satisfies the same conclusions of \cref{Proprin1}.
\end{theorem}

Of course, similar remarks can be done as in the previous Sections concerning equivalent notions of solutions as well as other minor features. For example, if $N=2$, since one is obliged to take $\sigma > \frac{2(q-1)}{2-q}$, the limiting case $q=\frac43$ has been included (compare with \cref{q=}) and the constants of the a priori estimates can be chosen to depend only on the norm of initial data: indeed,  bounded sets in $\elle\sigma$ are always equi-integrable in $\elle{\tilde\sigma}$ for any  $\frac{2(q-1)}{2-q}<\tilde\sigma <\sigma$.  

\section{A comparison result for linear operators}\label{uni}

In this Section we show that the a  priori bounds obtained so far also imply the uniqueness of solutions  in the weak formulation of \rife{P}. In order to avoid several minor cases, we only restrict to $N>2$. The cases $N=1,2$ can be done exactly in the same way according to the thresholds specified in the previous Section.\\
For simplicity, we only consider the case  of linear operators, namely
\begin{equation}\tag{$P_{lin}$}
\label{Plin}
\begin{cases}
\begin{array}{ll}
\ds  u_t -\text{div}(A(t,x) D u)=H(t,x, D u) & \ds \text{in}\ Q_T,\\
\ds  u=0  &\ds  \text{on}\ (0,T)\times \partial\Omega,\\
\ds  u(0,x)=u_0(x)  &\ds \text{in}\  \Omega.
\end{array}
\end{cases}
\end{equation}
In order to prove    a comparison result between subsolutions and supersolutions of \eqref{Plin}, we follow an idea of \cite{LePo}. 

\begin{proposition}\label{compa}
Assume that $A(t,x)$ is a  bounded and coercive matrix, and let $H(t,x,\xi)$ be a Carath\'eodory function   such that
$$
H(t,x,\xi)= H_1(t,x,\xi)+ H_2(t,x,\xi)
$$
where 
\begin{eqnarray}
&\ds\xi\mapsto H_1(t,x,\xi)\qquad \hbox{is convex}
& \label{h11} 
\\
[2mm]
\ds
& H_1(t, x,\xi) \leq \gamma(1+ |\xi|^q)\qquad \hbox{
 with $1< q<2$}  
&\label{h12}
\end{eqnarray}
and $  H_2(t,x,\xi)$ satisfies 
\begin{eqnarray}
& |H_2(t,x,\xi)- H_2(t,x,\eta)|   \leq L |\xi-\eta|  & \label{h21}
\\
[2mm]
& H_2(t,x,\xi)- (1-\vep)H_2\left(t,x,\frac{\xi}{1-\vep}\right)  \leq 0 & \label{h22} 
\end{eqnarray}
for a.e. $(t,x)\in Q_T$, for every $ \xi,\eta \in \R^N$.

Let $u_{0}, v_0$ satisfy \eqref{ID1} (if $2-\frac N{N+1}<q<2$) or \rife{ID2} (if $q<2-\frac N{N+1}$),  and suppose $u$ and $v$ be a subsolution, respectively a supersolution, of \eqref{Plin} in the sense of \cref{def1} (if  $2-\frac N{N+1}<q<2$) or \cref{defrin2} (if   $q<2-\frac N{N+1}$). Then, if $u_0\leq v_0$, we have $u\le v$ in $Q_T$.

Finally, if $q= 2-\frac N{N+1}$, the same conclusion holds  if $u_{0}, v_0\in \elle\sigma$ for some $\sigma>1$ and $u,v$ be subsolution, respectively supersolution, of \eqref{Plin} in the sense of  \cref{def1}. 
\end{proposition}
\proof
Define $v_{\eps}=(1-\eps)v$. Thus, $v_{\eps}$ satisfies: 
\[
 (v_{\eps})_t- \text{div}(A(t,x) D v_{\eps}) \ge(1-\eps)H(t,x, D v).
\]
Setting $z_{\eps}=u-v_{\eps}$, we deduce that  
\begin{equation}\label{z}
(z_{\eps})_t-\text{div}(A(t,x) D z_{\eps})\le H(t,x, D u)-(1-\eps)H(t,x, D v)\,.
\end{equation}
Since the convexity assumption implies that
\[
H_1(t,x,p)\le (1-\eps)H_1(t,x,r)+\eps H_1\biggl(t,x,\frac{p-(1-\eps)r}{\eps}\biggr)\quad \forall\, p,r\in \mathbb{R}^N,
\]
we estimate \eqref{z} as follows:
$$
(z_{\eps})_t-\text{div}(A(t,x) D z_{\eps})\le \eps H_1\left(t,x,\frac{ D z_\eps}\eps\right)+  H_2(t,x, D u) -(1-\eps)H_2(t,x, D v)\,.
$$
Using \rife{h21}-\rife{h22} and since $v= \frac{v_\eps}{1-\eps}$,  we have
$$
H_2(t,x, D u) -(1-\eps)H_2(t,x, D v) \leq H_2(t,x, D u)- H_2(t,x,Dv_\eps) \leq L | D z_\eps|\,.
$$
Therefore we conclude that
$$
(z_{\eps})_t-\text{div}(A(t,x) D z_{\eps})\le \eps H_1\left(t,x,\frac{ D z_\eps}\eps\right)
+ L | D z_\eps|\,,
$$
which implies, on account of \rife{h12}, the inequality below:
\begin{equation}\label{hat}
(\hat{z}_{\eps})_t-\text{div}(A(t,x) D \hat z_{\eps})\le \gamma \left(| D \hat z_{\eps}|^q + 1\right)+L  | D \hat z_{\eps}|\,
\end{equation}
where we have set $\hat {z}_{\eps}=\frac{z_\eps}{\eps}$. By Young's inequality, we get
$$
(\hat{z}_{\eps})_t-\text{div}(A(t,x) D \hat z_{\eps})\le \left(\gamma+\frac1q\right) | D \hat z_{\eps}|^q +  \left(\gamma+ \frac{L^{q'}}{q'}\right) \,.
$$
Hence, setting $c_0=\left(\gamma+\frac1q\right)$ and $c_1=  \left(\gamma+ \frac{L^{q'}}{q'} \right)$, we have found that the function $\psi_\eps(t):= \hat{z}_{\eps}- c_1 t$ satisfies
$$
(\psi_{\eps})_t-\text{div}(A(t,x) D \psi_{\eps})\le c_0 | D \psi_{\eps}|^q \,.
$$
In particular, we deal with subsolutions of \eqref{P}. We notice that, at fixed $\eps>0$,  $\psi_\eps(\cdot)$ belongs to $\continue{\sigma}\cap \paraccaloc$; moreover, the function $\psi_\eps^+(\cdot)$ enjoys the same regularity and is still a subsolution. In addition, we have $\psi_\eps^+(0) \leq v_0^+$, hence $\psi_\eps^+(0)$ is bounded in $\elle{\sigma}$ independently of $\vep$. By applying  \cref{sap} (if $2-\frac N{N+1}<q<2$) we deduce that there  exists  a constant $M$, independent of $\vep$, such that 
\[
\|[\psi_\eps(t)]^+\|_{L^{\sigma}} \le M\quad\forall\,\, t >0.
\]
The same holds true with $\sigma=1$ if $q<2-\frac N{N+1}$ thanks to \cref{estL1}.
Recalling the definition of $\psi_\eps(\cdot)$, the previous inequality implies
\[
\|[u(t)-(1-\eps)v(t)]_+\|_{L^{\sigma}} \le \eps [M + c_1 t \, |\Omega|^{\frac1{\sigma}}] \quad\forall t\ge 0
\]
which, letting $\eps\to 0$, leads to the conclusion. 

Finally, in the case $q=2-\frac N{N+1}$ we assumed that $u_0,v_0\in \elle\sigma$ for {\it some} $\sigma>1$. This is enough to reduce to the previous case since if \rife{h12} holds with $q=2-\frac N{N+1}$ then it  obviously holds for all bigger $q$.
\endproof

We conclude by pointing out  an easy class of nonlinearities $H(t,x,Du)$ for which the previous comparison principle applies.

\begin{corollary}\label{uniq-cor} Assume that $A(t,x)$ is a  bounded and coercive matrix, and  let $H(t,x,\xi)$ be a Carath\'eodory function which is $C^2$ with respect to $\xi$ (uniformly for $(t,x)\in Q_T$) and satisfies  
\begin{eqnarray}
&\exists R>0:\qquad \xi\mapsto H (t,x,\xi)\qquad \hbox{is convex for $|\xi|>R$}
& \label{H11} \\
[2mm]
&
H(t, x,\xi) \leq \gamma(1+ |\xi|^q)\qquad \hbox{
 with $1< q<2$.}  & \label{H12}
 \end{eqnarray}
Then the conclusion of \cref{compa} is true and, in particular, if $u_0$ satisfies \rife{ID1} or \rife{ID2}, problem \rife{Plin} has a unique weak solution.  
\end{corollary}

\proof As suggested in \cite{LePo}, it is enough to write $H= H_1+ H_2$, where $H_1(t,x,\xi)= H(t,x,\xi)+ \mu \sqrt{1+|\xi|^2}$ and $H_2(t,x,\xi)=- \mu \sqrt{1+|\xi|^2}$, choosing $\mu$ sufficiently large so that $H_1$ be a globally convex function of $\xi$.
\endproof

\section{On the optimality of the Lebesgue class of initial data}\label{optimalsec}

We conclude our analysis by discussing the optimality of condition  \eqref{ID1} on the initial data. We are actually going to see that similar results as proved in  \cref{finen} can not be obtained if $u_0$ belongs to subcritical Lebesgue spaces. Roughly speaking, this means that assuming $u_0$ in a Lebesgue space $L^{\ro}(\Omega)$ with $\ro<\si$  does not allow \eqref{P} to admit a  solution, at least in the suitable class. \\

For the sake of simplicity, let us consider the problem \eqref{P} with the Laplace operator and the $q$ power of the gradient in the r.h.s.:
\begin{equation}\label{D}\tag{$\overline{P}$}
\begin{cases}
\begin{array}{ll}
\ds u_t-\D u=\gamma |D u|^q &\ds \text{in }Q_T,\\
\ds u=0 &\ds \text{on }(0,T)\times \partial \Omega,\\
\ds u(0,x)=u_0(x)  &\ds  \text{in } \Omega.
\end{array}
\end{cases}
\end{equation}
Our goal is to show that there exists 
\begin{equation}\label{IDNE}\tag{$\overline{ID1}$}
u_0\in L^{\ro}(\Omega)\quad\text{with}\quad\ro<\si,
\end{equation}
such that  the problem \eqref{D}  does not admit a (renormalized) solution $u$ such that 
\begin{equation}\label{solne1}
u\in  L^\infty(0,T;L^{\ro}(\Omega)),\,\,|u|^\frac{\ro}{2}\in L^2(0,T;H_0^1(\Omega)).
\end{equation}
We recall that this class is natural in view of the summability of the initial datum and, in particular, existence and uniqueness in this class have been proved if $\rho \geq \si$, see \cite{M} (for the existence) and  \cref{uni} (for the uniqueness) respectively. Therefore, the counterexample  below  is an evidence of the optimality of the above results.  

Here  we follow the lines of  \cite[Subsection $3.2$]{BASW}, where  a similar counterexample was proved for the case of $\Omega= \R^N$ and in a  different formulation.

\begin{theorem}
Let $1\le \ro<\sigma$. There exists a value $\de>0$ (only depending on $\rho, q, N$) such that  if $u_0(x)\geq |x|^{-\frac{N}{\ro}+\delta}\chi_{|x|<1}$, then \eqref{D} does not admit any renormalized solution (in the sense of \rife{sr1}-\rife{sr2} of \cref{defrin}) such that \eqref{solne1} hold.
\end{theorem}
\vskip0.3em
In particular, we deduce that for some $u_0\in   L^{\ro}(\Omega)$, where $1\le \ro<\sigma$, no solution can exist of problem \rife{D}.
\vskip0.3em
\begin{proof}
Suppose, by contradiction, that there exists a solution $u$ as in \rife{sr1}--\rife{sr2} and satisfying \eqref{solne1}, under the assumption that  $\ro>1$. The case $\ro=1$ can be dealt with in  a similar  way.\\

\noindent
{\it Step 1} \\
\noindent
Note that such a solution will also satisfy that 
\[
\iint_{Q_T} |D u|^q|u|^{\ro-1}\,dx\,dt<\infty,
\]
which implies 
\begin{equation}\label{solne2}
|u|^{\frac{\ro+q-1}{q}}\in L^q(0,T;W^{1,q}_0(\Omega))\,.
\end{equation}
Roughly speaking, the boundedness of the above integral follows by taking $|u|^{\rho-1}$ in the renormalized formulation in $(t_0,T)$, obtaining that
$$
\gamma \int_{t_0}^T\into  |D u|^q|u|^{\ro-1}\,dx\,dt \leq (\rho-1)\iint_{Q_T} | D u|^2  |u|^{\ro-2}dx\,dt+ \frac1{\rho} \|u\|_{L^\infty(0,T;L^{\ro}(\Omega))}
$$
and then letting $t_0\to 0$. Of course, a  standard procedure is followed in order to justify the test function in the renormalized formulation (as in the proof of \cref{equivren}); in addition,  when $\rho<2$, one suitably modifies the test function near $u=0$ in an obvious way. \\

\noindent
{\it Step 2} \\
\noindent
Let $t$ be small and fixed. The estimates of the heat kernel in a bounded domain (see e.g.  \cite[Theorem $1.1$]{Z}) imply that the solution of the corresponding heat problem 
$$
\begin{cases}
\begin{array}{ll}
\ds U_t-\D U=0 &\ds \text{in }Q_T,\\
\ds U=0 &\ds \text{on }(0,T)\times \partial \Omega,\\
\ds U(0,x)=u_0(x)  &\ds  \text{in } \Omega.
\end{array}
\end{cases}
$$
satisfies the following bound from below:
\begin{align*}
U(t,x)\ge c_1\int_{\{\frac{\sqrt{t}}{2}<|y|<\sqrt{t}\}}t^{-\frac{N}{2}}e^{-c_2\frac{|x-y|^2}{t}}|y|^{-\frac{N}{\ro}+\delta}\,dy\ge ct^{-\frac{N}{2\ro}+\frac{\delta}{2}}\qquad \forall x\,: |x|\leq \sqrt t\,,
\end{align*}
for some positive constants $c_1$, $c_2$. Then, this estimate yields 
\begin{equation}\label{U}
\int_{\{|x|<\sqrt{t}\}} U(t,x)\,dx\ge c\, t^{\frac{N}{2}(1-\frac{1}{\ro})+\frac{\delta}{2}}
\end{equation}
where, here and below, $c$ denotes possibly different constants independent of $\de$, $t$.
\\
\noindent
We now look for a bound from above for a suitable integral of the solution of \eqref{D}. The boundedness \eqref{solne2} ensures that there exists a sequence $\{t_j\}_j$, converging to zero for $j\to \infty$, such that
\begin{equation}\label{tj}
\int_{\Omega}|D u(t_j)|^q|u(t_j)|^{\ro-1}\,dx=\|D (|u(t_j)|^{\frac{\ro+q-1}{q}})\|_{L^q(\Omega)}^q\le \frac{1}{t_j}.
\end{equation}
By H\"older's inequality with indices $\left(q^*\frac{\ro+q-1}{q}, \left(q^*\frac{\ro+q-1}{q} \right)'\right)$ and \eqref{tj} we go further estimating as below:
\begin{align*}
\int_{\{|x|<\sqrt{t_j}\}}u(t_j,x)\,dx &\le \|u(t_j)\|_{L^{q^*\frac{\ro+q-1}{q}}(\Omega)}
\meas\left\{|x|<\sqrt{t_j}\right\}^{1-\frac{q}{q^*(\ro+q-1)}}\\
&\le c\|D (|u(t_j)|^{\frac{\ro+q-1}{q}})\|_{L^q(\Omega)}^{\frac{q}{\ro+q-1}}\,
t_j^{\frac{N}{2}\left(1-\frac{q}{q^*(\ro+q-1)}\right)}\\
&\le c t_j^{-\frac{1}{\ro+q-1}+\frac{N}{2}\left(1-\frac{N-q}{N(\ro+q-1)}\right)}\\
&=ct_j^{\frac{N}{2}\left(1-\frac{N+2-q}{N(\ro+q-1)}\right)}.
\end{align*}

\noindent
{\it Step 3} \\
\noindent
In conclusion, Step $2$ provides us with the inequality
\begin{equation*}
\int_{\{|x|<\sqrt{t_j}\}}u(t_j,x)\,dx\le ct_j^{\frac{N}{2}\left(1-\frac{N+2-q}{N(\ro+q-1)}\right)}
\end{equation*}
while standard comparison results for the heat equation ensure that
\[
\int_{\{|x|<\sqrt{t}\}}u(t,x)\,dx\ge\int_{\{|x|<\sqrt{t}\}} U(t,x)\,dx\ge c\, t^{\frac{N}{2}\left(1-\frac{1}{\ro}\right)+\frac{\delta}{2}}.
\]
Comparing those inequalities we deduce that
\be\label{tj2}
c\, t_j^{\frac{N}{2}\left(1-\frac{1}{\ro}\right)+\frac{\delta}{2}} \leq \int_{\{|x|<\sqrt{t_j}\}}u(t_j,x)\,dx\le ct_j^{\frac{N}{2}\left(1-\frac{N+2-q}{N(\ro+q-1)}\right)}.
\ee
Since $\rho<\sigma$, there exists $\de>0$ (only depending on $\rho$ and $\si$) such that  
$$
1-\frac{1}{\ro}< 1-\frac{N+2-q}{N(\ro+q-1)}-\frac \de N\,.
$$
But this implies that 
$$
\frac{N}{2}\left(1-\frac{1}{\ro}\right)+\frac{\delta}{2} <\frac{N}{2}\left(1-\frac{N+2-q}{N(\ro+q-1)}\right)
$$
and letting $t_j\to 0$ in \rife{tj2} we get a  contradiction.
\end{proof}

\begin{remark}
The renormalized formulation of the equation is only one possible choice of weak formulation for which  a similar nonexistence result can be proved. We could have equally replaced it by asking that $u\in L^2_{loc}((0,T);H^1_0(\Omega))\cap \continue{\rho}$, $u(0)=u_0$ in $\elle{\rho}$,  and the equation be understood in distributional sense. The above nonexistence result would still hold in this formulation and the proof be the same as before. 
\end{remark}


%


\appendix

\section{}\label{app}

Here we present the proof of the equivalence between possibly different notions of solution of \rife{P}.
\vskip1em

\begin{proof}[Proof of \cref{defequiv}.]
Let $u$ be a solution according to \cref{def1}. We  already proved in Step 1 of \cref{lemGk} that  $(1+| u|)^{\frac{\si}{2}-1}u\in L^2(0,T; H^1_0(\Omega))$, i.e. $u$ satisfies \rife{beta}. Then, since $\sigma\geq 2$ and \rife{A2} holds,   it becomes a standard fact that the formulation can be extended up to $t=0$ and so $u$ is also a solution according to \cref{def2}.

Conversely, we assume now that $u$ is a solution according to \cref{def2} and we wish to prove that it also satisfies \cref{def1}. 
To this purpose, we use  $\psi(u)=|T_n(G_k(u))|^{\sigma-2}T_n(G_k(u))$ as test function in the interval $(0,t)$. Notice that this is allowed because $u\in L^2(0,T; H^1_0(\Omega))$ and in this case the conclusion of \cref{propint} is true up to $t_1=0$.
Thanks to the assumptions \eqref{A1} and to \eqref{H} we have:
\beac\label{c}
\ds\int_{\Omega}S_n(G_k(u(t)))\,dx+
\alpha(\sigma-1)
\iint_{Q_t} | D T_n(G_k(u)) |^2 |T_n(G_k(u))|^{\sigma-2}\,dx\,ds\\
[3mm]\ds
\le\gamma  \iint_{Q_t} | D G_k(u)|^q |T_n(G_k(u))|^{\sigma-1}\,dx\,ds
+\int_{\Omega}S_n(G_k(u_0))\,dx
,
\eeac
\noindent
where $S_n(v)=\int_0^v |T_n(z)|^{\sigma-2}T_n(z)\,dz$.
We estimate the first integral in the r.h.s. using H\"older's inequality with indices $\left(\frac{2}{q}, \frac{2^*}{2-q},\frac{N}{2-q}\right)$ and then Sobolev's embedding, getting
\begin{equation}\label{a}
\begin{array}{c}
\ds
\iint_{Q_t} | D G_k(u)|^q |T_n(G_k(u))|^{\sigma-1}\,dx\,ds \\
[3mm]\ds
\le \frac{2^q}{\si^q}\int_0^t\biggl[\biggl(\int_{\Omega} | D [|G_k(u)|^{\frac{\si}{2}}]|^2 \,dx\biggr)^{\frac{q}{2}}\biggl(\int_{\Omega}  |T_n(G_k(u))|^{\frac{2^*}{2}\si}\,dx\biggr)^{\frac{2-q}{2^*}}
\biggl(\int_{\Omega}  |T_n(G_k(u))|^{\sigma}\,dx\biggr)^{\frac{2-q}{N}}\biggr]\,ds
\\
[3mm]\ds\le 2^q\frac{c_S}{\si^q}\int_0^t\biggl(\int_{\Omega} | D [|G_k(u)|^{\frac{\si}{2}}]|^2 \,dx\biggr)\biggl(\int_{\Omega}  |T_n(G_k(u))|^{\sigma}\,dx\biggr)^{\frac{2-q}{N}}\,ds.
\end{array}
\end{equation}
In particular, this means that the following inequality 
\begin{equation}\label{b}
\begin{array}{c}
\ds
\int_{\Omega}S_n(G_k(u(t)))\,dx\\
 [3mm]\ds
\le 2^q\frac{c_S\gamma}{\si^q}\int_0^T\biggl(\int_{\Omega} | D [|G_k(u)|^{\frac{\si}{2}}]|^2 \,dx\biggr)\biggl(\int_{\Omega}  |T_n(G_k(u))|^{\sigma}\,dx\biggr)^{\frac{2-q}{N}}\,ds+\int_{\Omega}S_n(G_k(u_0))\,dx
\end{array}
\end{equation}
holds. Since we have $|u|^{\frac{\si}{2}}\in L^2(0,T;H_0^1(\Omega))$ and $\int_{\Omega}S_n(G_k(u))\,dx\ge c\int_{\Omega} |T_n(G_k(u))|^{\sigma}\,dx $ where $c=c(N,q)$, the inequality in \eqref{b} implies that
\begin{equation*} 
\sup_{t\in(0,T)}\int_{\Omega}  |T_n(G_k(u(t)))|^{\sigma}\,dx\le c_0+ c\biggl(\sup_{t\in(0,T)}\int_{\Omega}  |T_n(G_k(u(t)))|^{\sigma}\,dx\biggr)^{\frac{2-q}{N}},
\end{equation*}
where the constants $c$ and $c_0$ depend on, respectively, the $L^2(0,T;H_0^1(\Omega))$ norm of $|u|^\frac{\si}{2}$ and on the $L^{\sigma}(\Omega)$ norm of the initial datum. Hence, we get a uniform bound in $n$ for $\int_{\Omega}|T_n(G_k(u(t)))|^{\sigma}\,dx$. Letting $n$ tend to infinity, we have proved that $G_k(u)\in L^{\infty}(0,T;L^{\sigma}(\Omega))$ which in turn implies $u\in L^{\infty}(0,T;L^{\sigma}(\Omega))$.\\
Moreover, as $n\to \infty$ in \eqref{b}, we get
\begin{equation}\label{gk0}
\int_{\Omega}|G_k(u(t))|^{\sigma}\,dx\le \int_{\Omega}|G_k(u_0)|^{\sigma}\,dx+c\int_0^T \int_{\Omega} | D [|G_k(u)|^{\frac{\si}{2}}]|^2 \,dx\,ds
\end{equation}
where $c$ is a constant due to the $L^{\infty}(0,T;L^{\sigma}(\Omega))$ norm of $G_k(u)$. Having $|u|^\frac{\si}{2}\in L^2(0,T;H_0^1(\Omega))$ the right-hand side in \rife{gk0} tends to zero as $k\to \infty$. Hence  $\int_{\Omega}|G_k(u(t))|^{\sigma}\,dx\to 0 $ for  $k\to \infty$ and this fact allows us to deduce the continuity regularity $C([0,T];L^{\sigma}(\Omega))$ of $u$. Indeed, let $\{t_n\}_{n}$ be a sequence such that $t_n\to t$ for $n\to \infty$, for $t\in [0,T]$. Then  $u(t_n)$ converges to $u(t)$ in $L^1(\Omega)$ when $n\to\infty$ by \cref{rmkct}. Moreover, $u(t)$ is (uniformly in $t$) equi-integrable in the $L^{\sigma}(\Omega)$ norm. To convince us in this sense, we estimate
\begin{align*}
\sup_{t\in (0, T)}\int_E |u(t)|^{\sigma}\,dx&\le c\sup_{t\in (0, T)}\biggl[\int_{E\cap \{|u(t,x)|>k\}} |u(t)|^{\sigma}\,dx+
\int_{E\cap \{|u(t,x)|\le k \}} |u(t)|^{\sigma}\,dx\biggr]\\
&\le c \biggl[\sup_{t\in (0,T)}\int_{\Omega} |G_k(u(t))|^{\sigma}\,dx+k^{\sigma}|E|\biggr]
\end{align*}
where $c=c(N,q)$.
Thus, letting $|E|\to 0$ and then $k\to \infty$, and using \rife{gk0}, we have the desired equi-integrability condition verified. Finally, by Vitali's theorem, we deduce that $u(t_n)\to u(t)$ in $\elle\sigma$. Hence $u\in \continue {\sigma}$, and therefore it is also a  solution in the sense of \cref{def1}.
\end{proof}

\vskip1em
We  now give the proof of \cref{equivren}. To this purpose, we 
define   the function $\te_n(\cdot)$ as
\begin{equation}\label{suppcmpt}
\theta_n(v)=
\begin{cases} 
\begin{array}{ll}
\ds 1&\ds  |v|\le n,\\
\ds \frac{2n-|v|}{n}&\ds  n<|v|\le 2n,\\
\ds 0 &\ds |v|>2n
\end{array}
\end{cases}
\end{equation}
and we notice  that it is a compactly supported function which converges to $1$ as $n\to \infty$. This is  currently used in the renormalized formulation in order to recover, asymptotically, the case that  the auxiliary function $S$ in \rife{sr2} may have non compact support.

\vskip1em

\begin{proof}[Proof of \cref{equivren}.]

 Let $u$ be a solution according to \cref{def1}. From \cref{lemGk}, Step 1, we already know that $u$ satisfies \rife{beta}. 
Then,  $(1+ |u|)^{\frac\sigma2} \in \limitate2 \cap L^2(0,T; L^{2^*}(\Omega))$, which implies, by interpolation, that $u\in L^{\sigma \frac{N+2}N}(Q_T)$. Using that, by Young's inequality,  
$$
|Du|^q \leq c\left[|D((1+ |u|)^{\frac\sigma2})|^2 + (1+ |u|)^{\frac{(2-\sigma)q}{2-q}}\right]   
$$
 we deduce that $|Du| \in L^q(Q_T)$ since $\frac{(2-\sigma)q}{2-q}\leq \sigma \frac{N+2}N$. Hence $H(t,x,Du)\in L^1(Q_T)$. Moreover, as a direct consequence of \rife{beta}, we have that $T_k(u)\in L^2(0,T; H^1_0(\Omega))$ for every $k>0$.
Now it is enough to observe that, being $u\in L^2_{loc}(0,T; H^1_0(\Omega))$,  the renormalized formulation  holds in a standard way for $t\in (\tau, T)$, $\tau >0$. This means that  \rife{sr2} holds in $Q_{\tau,T}:= (\tau,T)\times \Omega$, and letting $\tau \to 0$ is allowed, providing with \rife{sr2} in $Q_T$. So \cref{defrin} is proved to hold.
 \vskip0.5em
Conversely, let $u$ be a solution according to \cref{defrin}. We wish to prove that solutions belong to $\continue\si$ and become regular at positive time. 

First of all, we remark that,  by a density argument,  the class of test functions can be extended to include any $\vfi\in L^2(0,T;H_0^1(\Omega))\cap  L^{\infty}(Q_T)$ such that $\vp_t\in L^{2}(0,T;H^{-1}(\Omega))$.  We take  $S'(u) =\int_0^{T_n(G_k(u))}(\eps+|z|)^{\sigma-2}\,dz$, $\vp=1$ and $\eps>0$, in \eqref{sr2} for $0\le t\le T$ and say that this step can be made rigorous thanks to the regularity \eqref{beta} and by standard arguments in the renormalized framework (for instance, we can consider  $S'(u)=\te_n(u)\int_0^{T_n(G_k(u))}(\eps+|z|)^{\sigma-2}\,dz$ where $\te_n(\cdot)$ is defined in \eqref{suppcmpt}, and let $n$ go to infinity).

We proceed recalling \eqref{A1} and \eqref{H} which give us the following inequality:
\be
\begin{split}\label{uf}
& \int_{\Omega} \Theta_\eps(u(t))\,dx+\al\iint_{Q_t}|D T_n(G_k(u))|^2[\eps+|T_n(G_k(u))|]^{\sigma-2} \,dx\,ds\\
& \quad \le \int_{\Omega} \Theta_\eps(u_0)\,dx+\gamma  \iint_{Q_t} | D G_k(u)|^q\biggl(\int_0^{T_n(G_k(u))}(\eps+|z|)^{\sigma-2} \,dz\biggr)\,dx\,ds
\end{split}
\ee
where we have set $\displaystyle\Theta_\eps(v)=\int_0^v \left(\int_0^{T_n(G_k(z))} (\eps+|s|)^{\sigma-2}\,ds \right)\,dz$.

Now we can estimate
\begin{equation*}
\begin{array}{c}
\ds
 \iint_{Q_t} | D G_k(u)|^q\biggl(\int_0^{T_n(G_k(u))}(\eps+|z|)^{\sigma-2} \,dz\biggr)\,dx\,ds\\
[3mm]\ds
\leq 
  c \iint_{Q_t} | D G_k(u)|^q (\eps+|T_n(G_k(u))|)^{\sigma-1} \,dx\,ds
\\
[3mm]\ds
\leq c \iint_{Q_t} | D G_k(u)|^q (\eps+|T_n(G_k(u))|)^{(\sigma-2)\frac q2}  (\eps+|T_n(G_k(u))|)^{\frac\sigma 2(2-q)}\, (\eps+|T_n(G_k(u))|)^{q-1}dx\,ds\\
[3mm]\ds
 \leq 
c  \sup\limits_{s\in [0,t]} \left[\into (\eps+|T_n(G_k(u))|)^{\sigma}dx\right]^{q-1}
\end{array}
\end{equation*}
where we used again the H\"older inequality with indices $\left(\frac{2}{q},\frac{2^*}{2-q},\frac{N}{2-q}\right)$ (recall that $\sigma= \frac{N(2-q)}{q-1}$) and Sobolev's embedding, and the information that $u$ satisfies \eqref{beta} (so the constant $c$ in the last inequality depends on the $L^2(0,T;H_0^1(\Omega))$ norm of $(1+|u|)^\frac{\si}{2}u$). Finally,
from \rife{uf} we get
\begin{equation*}
\begin{array}{c}
\ds \int_{\Omega} \Theta_\eps(u(t))\,dx  +\al\iint_{Q_t}|D T_n(G_k(u))|^2[\eps+|T_n(G_k(u))|]^{\sigma-2} \,dx\,ds 
\\
[3mm]\ds\le \int_{\Omega} \Theta_\eps(u_0)\,dx+c  \sup\limits_{s\in [0,t]} \left[\into (\eps+|T_n(G_k(u))|)^{\sigma}dx\right]^{q-1}\,.
\end{array}
\end{equation*}
Since $q-1<1$, this inequality implies first that $T_n(G_k(u))$ is uniformly bounded (with respect to $n$) in $\limitate\sigma$. Hence we deduce that $u\in \limitate\sigma$ and,   taking the limit as $n\to \infty$ in the previous inequality, we obtain  an inequality of the type \eqref{gk0}. This means that  we can reason as in \cref{defequiv} to deduce the continuity regularity $u\in C^0([0,T];L^{\sigma}(\Omega))$.
Finally, we are only left with the fact that solutions have finite energy after any positive time. This is the content of \cref{renreg} below and concludes this proof.
\end{proof}

\begin{lemma}\label{renreg} Let $q>2-\frac N{N+1}$ and $u_0\in \elle\sigma$. Then, if $u$ is a renormalized solution according to \cref{defrin}, we have $u\in L^2_{loc}(0,T;H_0^1(\Omega))$ and, for any $r\in (\sigma,2]$ and $t\leq 1$,   
 $$
 \|u(t)\|_{L^{r}(\Omega)}\le  c\, t^{- N\frac{(r-\si)}{2r\si}}
 $$
with $c= c(\alpha, \gamma, q,N, u_0, |\Omega|)$. In particular, $u$ satisfies the conclusion and the estimates, both \rife{gkr} and \rife{decayu}, of \cref{Prop2r}.
 \medskip
 
Finally, if $q<2-\frac N{N+1}$ and $u_0\in \elle1$, any $u$ which is a  renormalized solution according to \cref{defrin2} satisfies the same conclusion with $\sigma=1$ and $r>1$.
 \end{lemma}

\proof  We first deal with the case $q>2-\frac N{N+1}$ and $u_0\in \elle\sigma$, $\sigma= \frac{N(2-q)}{q-1}$. 

Let $S'(u)\vp$ in \eqref{sr2} with $\vp=1$, $S\in W^{2,\infty}(\mathbb{R})$ such that $S'(0)=0$ and  
\begin{equation}\label{S1}
0\le S''(x)\le L(1+|x|)^{\sigma-2}.
\end{equation}
Then we claim that the following differential inequality holds:
\begin{equation}\label{w11u}
\frac{\text{d}}{\text{d}t}\int_{\Omega}S(u(t))\,dx+\al\int_{\Omega}|D u|^2 S''(u)\,dx\le \ga\int_\Omega |D u|^q\, |S'(u)|\,dx\qquad \text{a.e.}\,\, t\in(0,T).
\end{equation}
Again, since our choice of $S'(\cdot)$ is not an admissible one, we prove this claim taking $\xi\,\tS'_n(u)=\xi S'(u)\te_n(u)$ in \eqref{sr2} where $\xi\in C_c^\infty(0,T)$ and $\te_n(\cdot)$ as in \eqref{suppcmpt}. Then, we obtain
\begin{equation*}
\begin{array}{c}
\ds
-\int_0^T\xi'\integrale \tS_n(u)\,dx\,ds\\
[3mm]\ds
\le -\al\int_0^T\xi\int_{\Omega}|D u|^2\left[ S''(u) \te_n(u) +S'(u)\te_n'(u)\right]\,dx\,ds+\ga \int_0^T\xi\int_{\Omega}|D u|^q\, |S'(u)|\,dx\,ds.
\end{array}
\end{equation*}
The r.h.s. is uniformly bounded in $n$ thanks to \eqref{beta} and \eqref{S1} (in particular, $|Du|^q |S'(u)|\leq c |Du|^q (1+|u|)^{\sigma-1}$ which is $L^1(Q_T)$ because of \rife{beta} and $u\in \limitate\sigma$).

Moreover, \eqref{S1} implies that
\[
\iint_{Q_t} |D u|^2S'(u)\te_n'(u)\,dx\,ds
\le c\iint_{\{n<|u|<2n\}}|D ((1+|u|)^{\frac{\si}{2}-1}u)|^2\,dx\,ds\to 0
\]
by the definition of $\te_n(\cdot)$ in \eqref{suppcmpt}. Then we conclude that, letting $n\to\infty$, we recover \eqref{w11u}.\\
In particular, we have
\begin{equation}\label{w11rin}
\frac{\text{d}}{\text{d}t}\int_{\Omega}S(G_k(u(t)))\,dx+\al\int_{\Omega}|D G_k(u)|^2 S''(G_k(u))\,dx\le\ga\int_\Omega |D G_k(u)|^q\, |S'(G_k(u))|\,dx,
\end{equation}
since $S(G_k(s))$ verifies \rife{S1} whenever $S(\cdot)$ does. 
\medskip

We proceed requiring that the function $S(\cdot)$ satisfies
\begin{equation}\label{S2}
 |S'(x)| \le c\, \left(S''(x)\right)^\frac{q}{2} \, \left |\int_0^{x} \left(S''(y)\right)^\frac{2-q}{2}\,dy\right| \qquad \forall x\in \R
\end{equation}
for some $c>0$.  Since
\[
\left |\int_0^{G_k(u)} \left(S''(v)\right)^\frac{2-q}{2}\,dv \right | \le \left(\int_0^{|G_k(u)|} \left(S''(v)\right)^\frac{1}{2}\,dv\right)^{2-q}|G_k(u)|^{q-1}
\]
we get the inequality
\begin{equation*}
\frac{\text{d}}{\text{d}t}\int_{\Omega}S(G_k(u(t)))\,dx+\al \integrale |D \Phi(G_k(u))|^2\,dx\\\le c\integrale |D \Phi(G_k(u))|^q\, |\Phi(G_k(u))|^{2-q}\, |G_k(u)|^{q-1}\,dx
\end{equation*}
where $\Phi(x)=\int_0^x \left(S''(y)\right)^\frac{1}{2}\,dy$. This means that an application of H\"older's inequality with indices $\left(\frac{2}{q},\frac{2^*}{2-q},\frac{N}{2-q}\right)$ and Sobolev's embedding give us
\[
\frac{\text{d}}{\text{d}t}\int_{\Omega}S(G_k(u(t)))\,dx+\al \integrale |D \Phi(G_k(u))|^2\,dx\le c\sup_{t\in (0,T)}\|G_k(u(t))\|_{L^{\sigma}(\Omega)}^{q-1}\integrale |D \Phi(G_k(u))|^2\,dx
\]
and then, for $k$ sufficiently large  (eventually only depending on $u_0$)
\begin{equation}\label{phi}
\frac{\text{d}}{\text{d}t}\int_{\Omega}S(G_k(u(t)))\,dx+c \integrale |D \Phi(G_k(u))|^2\,dx\le 0.
\end{equation}
We now choose, for $r>\si$,  
$$
S'(\eta)= \int_0^\eta (\vep+|\xi|)^{\sigma-2} (\vep+ |T_n(\xi)|)^{r-\sigma}d\xi\,.
$$
We easily estimate, for constants $c_0, c_1$ independent of $\vep$ and $n$, 
$$
\begin{array}{c}
|S'(\eta)|\leq c_0 |\eta|(\vep+|\eta|)^{\sigma-2} (\vep+ |T_n(\eta)|)^{r-\sigma},\\
[3mm]\ds
 |S(\eta)| \leq c_1 |\eta|^2(\vep+|\eta|)^{\sigma-2} (\vep+ |T_n(\eta)|)^{r-\sigma}\,.
\end{array}
$$ 
Hence, interpolating $2=   2^*\omega+2\frac\si r (1-\omega)$ we have  
\besac 
\ds
\into S(G_k(u)) \,dx  \leq c_1  \into  | G_k(u)|^2(\vep+|G_k(u)|)^{\sigma-2} (\vep+ |T_n(G_k(u))|)^{r-\sigma}  \,dx
\\
[3mm]\ds
\leq c_1 \into \left[| G_k(u)| (\vep+|G_k(u)|)^{\frac{\si}{2}-1} (\vep+ |T_n(G_k(u))|)^{\frac{r- \si}{2} }\right]^{ 2^*\omega} 
  (\vep+|G_k(u)|)^{\sigma(1-\omega)}  \,dx
\\
[3mm]\ds \leq  c_1\left(\into   \left[| G_k(u)| (\vep+|G_k(u)|)^{\frac{\si}{2}-1} (\vep+ |T_n(G_k(u))|)^{\frac{r-\si}{2} }\right]^{ 2^*} \,dx\right)^\omega   \left(\into  \left[ \vep+|G_k(u)|\right]^{\si}\,dx\right)^{1-\omega}  
\\
[3mm]\ds \leq K \left(\into  \Phi(G_k(u))^{ 2^*} \,dx\right)^\omega
\eesac
for some $K= c\left[\vep+ \|G_k(u)\|_{\continue\si}\right]$. 
Recalling the value of $\omega$, Sobolev inequality and \rife{phi} we conclude that
$$
\frac{\text{d}}{\text{d}t}\int_{\Omega}S(G_k(u(t)))\,dx+c \left(\int_{\Omega}S(G_k(u(t)))\,dx\right)^{\gamma}\le 0
$$
where $\gamma= 1+ \frac{2\sigma}{N(r-\si)}$.  This implies, as in \cref{Prop2r} (see \rife{distheta}) that, for any $\tau>0$,  $\into S(G_k(u)(t))dx$ is bounded  in $(\tau,T)$ by some constant which is independent of $n$ and of $\vep$. As a first consequence, letting $n$ go to infinity we deduce that $u\in L^\infty_{loc}((0,T); \elle{r})$ for every $r\leq 2$. Indeed, as $n\to \infty$ the function $S(\cdot)$ becomes
$$
S(r)= \int_0^r \int_0^\eta (\vep+|\xi|)^{r-2}d\xi d\eta\,.
$$
In addition, when $r=2$, going back to \rife{w11rin}, integrating  and letting $n\to \infty$ we also  deduce  that $u\in \paraccaloc$.  
Finally, since the estimate was uniform with respect to $\vep$, we can let $\vep \to 0$ and we obtain the same precise estimate \rife{gkr} as in \cref{Prop2r}. Estimate \rife{decayu} will follow consequently. Notice that, once $u$ has been estimated in $\elle2$, then the estimate with in $\elle r$ with $r>2$  follows directly from \cref{Prop2r} again.
\medskip

Let us now deal with the case that $q<2-\frac N{N+1}$ and $u_0\in \elle1$. Here we follow the same steps as before except that at first we are only allowed to use a  function $S(\cdot)$ such that $S'(\cdot)$ is bounded. Therefore we choose
$$
S'(\eta)= \int_0^\eta (\vep+|\xi|)^{\de-2} (\vep+ |T_n(\xi)|)^{r-\de}d\xi\,
$$
with $\de<1$. After estimating in a similar way as above both $S'(\cdot)$ and $S(\cdot)$ we interpolate $2=   2^*\omega+\frac2 r (1-\omega)$ (with $r>1$) and we have  
\besac 
\ds
\into S(G_k(u)) \,dx  \leq c_1  \into  | G_k(u)|^2(\vep+|G_k(u)|)^{\de-2} (\vep+ |T_n(G_k(u))|)^{r-\de}  \,dx
\\
[3mm]\ds
\leq c_1 \into \left[| G_k(u)| (\vep+|G_k(u)|)^{\frac{\de-2}{2}} (\vep+ |T_n(G_k(u))|)^{\frac{r- \de}{2} }\right]^{ 2^*\omega} 
  (\vep+|G_k(u)|)^{(1-\omega)}  \,dx
\\
[3mm]\ds \leq   K \left(\into  \Phi(G_k(u))^{ 2^*} \,dx\right)^\omega
\eesac
for   $K= c\left[\vep+ \|G_k(u)\|_{\continue1}\right]$.  This way we obtain as before
$$
\frac{\text{d}}{\text{d}t}\int_{\Omega}S(G_k(u(t)))\,dx+c \left(\int_{\Omega}S(G_k(u(t)))\,dx\right)^{\gamma}\le 0
$$
where $\gamma= 1+ \frac{2}{N(r-1)}$ and the constant $c$ is independent of $n$ and $\vep$. As $n\to \infty$ and $\vep\to 0$ we find  $S(\eta)= \frac{|\eta|^r}{r(r-1)}$ and we deduce the desired estimate.
\endproof

\vskip1em
\begin{proof}[ Proof of \cref{equiv1}.]

Let  $u\in \continue1\cap \paraccaloc$ satisfy \rife{1loc}. We are going to show that it also verifies   \cref{defrin2}.\\
To this purpose, first of all we use \cref{anche} which says that $u\in {\mathcal T}^{1,2}_0(Q_T)$, $u\in M^{\frac{N+2}{N}}(Q_T)$ and $|Du|\in M^{\frac{N+2}{N+1}}(Q_T)$. In particular, this latter information implies that $H(t,x,Du)\in \parelle1$. Now,  by a standard finite energy argument,  the test function $S'(u)\vfi$ is allowed in $(\tau, t)$ for any $\tau >0$, where $S'\in W^{1,\infty}(\R)$, $S'$ has compact support and $S'(0) \vfi$ vanishes on $\partial \Omega$. On the other hand, the global informations that $u\in \continue 1$, that $T_k(u)\in L^2(0,T;H^1_0(\Omega))$ and that $H(t,x,Du)\in \parelle1$, allow us to let $\tau\to 0$ so that \rife{sr22} holds true.
Finally, recovering \rife{sr32} is standard in $L^1$ theory (see e.g. \cite[Theorem $2$]{Bl}, \cite[Lemma $3.2$ and Remark $2.4$]{BlMu}): in our setting, it is enough to use $(1-\theta_n(u)){\rm sign}(u)$ as test function in $(\tau, T)$, then letting $\tau\to 0$ one obtains the usual estimate
$$
\frac1n \iint_{\{n\le |u|\le 2n\}}a(t,x,u,D u)\cdot D u\,dx\,dt \leq \int_{\{|u_0|>n\}} |u_0| dx+   \iint_{\{  |u|>n\}} |H(t,x,Du)|dxdt
$$
which implies \rife{sr32}. Therefore, $u$ is a solution in the sense of  \cref{defrin2}.
\medskip

Conversely, assume now that $u$ is a solution satisfying  \cref{defrin2}. 
As far as the continuity of $u(t)$ in $L^1(\Omega)$ is concerned, let $\vp=1$ and $S'(u) =\frac{T_\omega(G_k(u))}{\omega}$,    $\omega>0$,  in \eqref{sr22}. We observe that such a test function can be made rigorous in the renormalized setting by taking $S'(u)=\frac{T_\omega(G_k(u))}{\omega}{\theta_n}(G_k(u))$ (which has compact support) and using  the asymptotic condition \eqref{sr32} to let $n\to \infty$.  Once we take $S'(u) =\frac{T_\omega(G_k(u))}{\omega}$, then integrating by parts the time derivative and letting   $\omega\to 0$ provide us with  the $L^1$ estimate
\begin{equation}\label{diss}
\integrale |G_k(u(t))\,dx\le M
\end{equation}
where 
\[
M=\integrale |G_k(u_0)|\,dx+\gamma\iint_{Q_T}|\N G_k(u)|^q\,dx\,dt.
\]
Since $|\N u|\in M^{\frac{N+2}{N+1}}(Q_T)$ (see \cref{estL1}) we deduce that $\|G_k(u(t))\|_{\elle1}$ tends to zero (uniformly in $t$) as $k\to \infty$. Since $T_k(u)\in C([0,T];L^1(\Omega))$ for every $k>0$ (see \cref{cont1}), it is easy to conclude that $u\in \continue1$. Finally, from \cref{renreg}, we have that $u\in \paraccaloc$. The equality \rife{1loc} now follows straightforwardly, by taking $\vfi\in C^\infty_c(Q_T)$ and $S'(u)= \theta_n(u) $ in \rife{sr22} and letting $n\to \infty$.
\end{proof}

\end{document}